\documentclass[12pt]{amsart}
\usepackage{amsmath,amsfonts,amsthm,amssymb,mathrsfs}
\usepackage{epsfig}
\usepackage{slashed}
\usepackage{mathtools}
\usepackage{enumerate}
\usepackage{fullpage}
 \usepackage{soul}
\usepackage{xcolor}
\usepackage{float}
\usepackage{hyperref}
\usepackage{cleveref}

\newcommand{\mb}{\mathbf}
\newcommand{\mc}{\mathcal}

\renewcommand{\Re}{\mathrm{Re}\,}

\newcommand{\rg}{\mathrm{rg}\,}
\newcommand{\N}{\mathbb{N}}
\newcommand{\R}{\mathbb{R}}
\newcommand{\C}{\mathbb{C}}

\newcommand{\B}{\mathbb{B}}

\DeclareMathOperator{\range}{rg}
\DeclareMathOperator{\rank}{rank}
\DeclareMathOperator{\tr}{tr}

\newcommand{\la}{\lambda}
\renewcommand{\r}{\rho}
\renewcommand{\d}{\delta}
\renewcommand{\t}{\tau}

\newtheorem{lemma}{Lemma}[section]
\newtheorem{theorem}[lemma]{Theorem}
\newtheorem{corollary}[lemma]{Corollary}
\newtheorem{proposition}[lemma]{Proposition}
\theoremstyle{remark}
\newtheorem{remark}[lemma]{Remark}

\theoremstyle{definition}
\newtheorem{definition}[lemma]{Definition}

\numberwithin{equation}{section}

\title[]{Singularity formation for the  higher dimensional Skyrme model in the strong field limit}

\author{Po-Ning Chen}
\address{University of California, Riverside, 900 University Ave, Riverside, CA 92521}
\email{poningc@ucr.edu}

\author{Michael McNulty}
\address{Michigan State University, 619 Red Cedar Road, East Lansing, MI 48824}
\email{mcnult50@msu.edu}

\author{Birgit Sch\"orkhuber}
\address{Universit\"at Innsbruck, Institut f\"ur Mathematik, Technikerstrasse 13, 6020 Innsbruck, Austria}
\email{Birgit.Schoerkhuber@uibk.ac.at}

\thanks{P.C. is supported by the Simons Foundation collaboration award \#584785.}

\begin{document}
\begin{abstract}
This paper concerns the formation of singularities in the classical $(5+1)$-dimensional, co-rotational Skyrme model. While it is well established that blowup is excluded in $(3+1)$-dimensions, nothing appears to be known in the higher dimensional case. 
We prove that the model, in the so-called strong field limit, admits an explicit self-similar solution which is asymptotically stable within backwards light cones. From a technical point of view, the main obstacle to this result is the presence of derivative nonlinearities in the corresponding evolution equation. These introduce first order terms in the linearized flow which render standard techniques useless. We demonstrate how this problem can be bypassed by using structural properties of the Skyrme model. 
\end{abstract} 

\maketitle

\section{Introduction}\label{Introduction}

In the early 1960s, physicist Tony Skyrme established his namesake model in nuclear physics \cite{S61a,S61b,S62} by introducing a higher-order correction term to the previously well-established nonlinear sigma-model for pions \cite{GL60}.  A natural extension of Skyrme's model for spatial dimensions $d \geq 3$, and maps $\Psi$ from Minkowski space $\mathbb R^{1+d}$ into the $d$-sphere $\mathbb S^d$, is described by the action functional \footnote{The Einstein summation convention of implicitly summing over repeated lower and upper indices is in effect.}
\begin{equation}
			\mc S_{Sky}[\Psi]=\alpha \mc S_{WM}[\Psi]+\frac{\beta}{4}\int_{\R^{1+d}}\Big(\big(\eta^{\mu\nu}(\Psi^*h)_{\mu\nu}\big)^2-(\Psi^*h)_{\mu\nu}(\Psi^*h)^{\mu\nu}\Big)d\eta \label{skyrme action}
		\end{equation}
where $\alpha, \beta \geq 0$, $\eta=\text{diag}(-1,1,\dots,1)$ denotes the Minkowski metric, $h$ is the standard round metric on $\mathbb S^d$, $(\Psi^*h)_{\mu\nu}=h_{ab}(\Psi)\partial_\mu \Psi^a\partial_\nu \Psi^b$ for $\mu,\nu=0,\dots,d$ and $a,b=1,\dots,d$, and 
\begin{equation}\label{wavemaps action}
 \mc S_{WM}[\Psi]=\frac{1}{2}\int_{\R^{1+d}}\eta^{\mu\nu}(\Psi^*h)_{\mu\nu}d\eta
\end{equation}
is the classical wave maps action which describes the nonlinear sigma-model. From a mathematical point of view,  geometric nonlinear field theories, such as those described by \eqref{skyrme action}, provide a rich source of challenging problems as the corresponding Euler-Lagrange equations entail highly non-trivial dynamical behavior.

We restrict our attention to so-called \textit{co-rotational} maps. These are maps $\Psi$ which, when expressed in spherical coordinates on its domain and co-domain, take the form
	\[
			\Psi(t,r,\omega)=\big(\psi(t,r),\omega\big).
	\]  
for some function $\psi:\R \times[0,\infty)\to \R$ and $\omega \in \mathbb S^{d-1}$. For such maps, the Euler-Lagrange equations for \eqref{skyrme action} yield a single radial \textit{quasilinear} wave equation
		\begin{align}
		\begin{split}
			\Big(\alpha&+\frac{\beta(d-1)\sin^2(\psi)}{r^2}\Big)\big(\partial_t^2\psi-\partial_r^2\psi\big)-\frac{d-1}{r}\Big(\alpha+\frac{\beta(d-3)\sin^2(\psi)}{r^2}\Big)\partial_r\psi
			\\
			&+\frac{(d-1)\sin(2\psi)}{2r^2}\bigg(\alpha+\beta\Big(\big(\partial_t\psi\big)^2-\big(\partial_r\psi\big)^2+\frac{(d-2)\sin^2(\psi)}{r^2}\Big)\bigg)=0. \label{skyrme eom}
		\end{split}
		\end{align}
We refer the reader to Appendix \ref{Derivation of the Equation} for the details of its derivation. 

By now, much is known for Equation \eqref{skyrme eom} in the case $d=3$, where the  model is famously known for admitting a soliton solution - the \textit{Skyrmion} - the existence of which has been proved in \cite{KL83,MT91}.  Its linear stability within the co-rotational class was established in \cite{CDSS16}, however, its full nonlinear asymptotic stability remains an open problem. Beyond that, there are several results addressing the Cauchy problem for Equation \eqref{skyrme eom}. In particular, global regularity for large data was established in \cite{GG18} and \cite{L21}. Global existence and scattering for small data in critical Sobolev-Besov spaces was established in \cite{GNR11}. For a comprehensive overview, we refer the reader to the monograph \cite{GG16}. To the best of the authors' knowledge, however, the case $d\geq 4$ appears entirely unexplored.
 
\subsection{The Skyrme model in the strong field limit} 
It is well-known that in the limiting case of Equation \eqref{skyrme eom} with $\beta = 0$, singularities can form in finite time in any dimension $d\geq 2$. More precisely, setting $\beta=0$ reduces Equation \eqref{skyrme eom} to the well-known wave maps equation
	\begin{equation}
		\partial_t^2\psi - \partial_r ^2\psi - \frac{d-1}{r} \partial_r \psi + \frac{(d-1)\sin(2\psi)}{2r^2} = 0\label{wave maps eom}
	\end{equation}
which has the explicit solution
\begin{equation}  \psi^{T}_{WM}(t,r) = 2\arctan\Big(\frac{r}{\sqrt{d-2}(T-t)}\Big), \label{wave maps soln}\end{equation}
for $d \geq 3$ (in the two dimensional case, blowup is more difficult to detect), see also Section \ref{Sec:related_results}. For $d=3$, adding the second term in \eqref{skyrme action} to the wave maps action prevents finite time blowup and allows for the existence of a nontrivial static solution. It appears unclear, however, whether or not Skyrme's `fix' to the wave maps action actually continues to prevent singularities from forming in higher space dimensions. 

Notice that the wave maps part of the action functional \eqref{skyrme action} is quadratic in the derivatives of $\Psi$ whereas the terms attached to $\beta$, which will be referred to as the \textit{strong field} part, are quartic. In particular, one might expect that for solutions with large gradients, the wave maps part becomes less relevant and dynamics are eventually governed by the Euler-Lagrange equation corresponding to $\alpha = 0$ which reads
\begin{align}
		\begin{split}
			\frac{\sin^2(\psi)}{r^2}\Big(\partial_t^2\psi & -\partial_r^2\psi-\frac{d-3}{r}\partial_r\psi\Big)
			\\
			&+\frac{\sin(2\psi)}{2 r^2}\bigg(\big(\partial_t\psi\big)^2-\big(\partial_r\psi\big)^2+\frac{(d-2)\sin^2(\psi)}{r^2}\bigg)=0. \label{sf skyrme eom}
\end{split}
\end{align}
We call Equation \eqref{sf skyrme eom} the equation of motion of the \textit{co-rotational, strong field Skyrme model}.

A few observations are in order. First, a direct calculation shows that solutions of Equation \eqref{sf skyrme eom} formally conserve the energy-type quantity
\[
			E_{SF}[\psi](t):=\frac{1}{2}\int_0^\infty\frac{\sin^2(\psi(t,r))}{r^2}\bigg((\partial_t\psi(t,r))^2+(\partial_r\psi(t,r))^2+\frac{d-2}{2}\frac{\sin^2(\psi(t,r))}{r^2}\bigg)r^{d-1}dr. \label{sf skyrme energy}
\]
	Furthermore,  in contrast to the full Skyrme model, Equation \eqref{sf skyrme eom} is scale invariant in the sense that given a solution $\psi$ and $\la>0$, one can obtain another solution $\psi_\la$ by setting 
		\begin{equation}
			\psi_\la(t,r) =\psi(t/\la,r/\la). \label{rescaling}
		\end{equation}
	The energy of a rescaled solution relates to that of the original solution according to
		$$
			E_{SF}[\psi_\la](t)=\la^{d-4}E_{SF}[\psi](t/\lambda). 
		$$
		The standard heuristic suggests that for $d \geq 5$, finite-time blowup via shrinking of solutions is energetically favorable.
In fact, for $d=5$ the second author \cite{M20} established the existence of a self-similar solution which is smooth in a backward light cone by using variational arguments al\'a Shatah \cite{S88}. Remarkably, we find that an \textit{explicit} self-similar solution exists in any dimension $d \geq 5$ which is given by 
\begin{align}
			\psi_{SF}^T(t,r)=U\Big(\frac{r}{T-t}\Big),\quad  T>0 \label{sf skyrme soln}
\end{align}
with the profile 
\begin{equation}\label{Profile_U_d}
 U(\rho)=\arccos\left (\frac{a- b \rho^2}{a+\rho^2} \right )
\end{equation}
where $a:=\frac{1}{3} \left(2 (d-4)+\sqrt{3(d-4) (d-2)}\right)$
and
$b :=2 \sqrt{\frac{d-4}{3(d-2)}}+1$. Observe that $U$ is smooth for $\rho \in [0, \rho^*]$, where $\rho^* = \sqrt{\frac{2a}{b-1}} > 1$. Moreover, $U(\rho^*) = \pi$. Hence,  $\psi_{SF}^T$ is a classical solution of Equation \eqref{sf skyrme eom} for $t\in(0,T)$ and $0\leq r \leq \rho^* (T-t)$. Moreover, while $ \psi^T_{SF}$ is perfectly smooth inside the backward light cone
\[ \mc C_T:=\{(t,r):0 \leq t<T,0\leq r<T-t\}, \]
it suffers a gradient blowup at the origin as $t \to T^{-}$ since 
\[  |\partial_r\psi_{SF}^T(t,0)|=\frac{c_d}{T-t}\]
for some $c_d > 0$. 

\subsection{The main result}\label{The Main Result}
We restrict ourselves to the lowest energy supercritical dimension $d=5$ and prove the stability of the self-similar blowup solution \eqref{sf skyrme soln} under small co-rotational perturbations, localized to a backward light cone, under the flow of Eq.~\eqref{sf skyrme eom}. For $d=5$, we have $a= b= \frac{5}{3}$ and the expression for the blowup profile \eqref{Profile_U_d} can be simplified to 
\begin{align}\label{Profile_U}
			U(\rho)=2\arctan\Big(\frac{2\rho}{\sqrt{5-\r^2}}\Big).
\end{align}

To state the main result, we slightly reformulate the problem.

\subsubsection{Refomulation as a nonlinear wave equation on $\R^{1+7}$}\label{Sec:Reformulation}

First, we observe that the self-similar solution satisfies $0\leq\psi_{SF}^T(t,r)<\pi$ for all $(t,r)\in\mc C_T$ with $\psi_{SF}^T(t,r)=0$ if and only if $r=0$. Assuming that $\psi$ is a smooth solution of Equation \eqref{sf skyrme eom} satisfying this same property, then Equation \eqref{sf skyrme eom} reduces to the following semilinear wave equation
		
			$$
				\Big(\partial_t^2\psi-\partial_r^2\psi-\frac{2}{r}\partial_r\psi\Big)+\cot(\psi)\Big(\big(\partial_t\psi\big)^2-\big(\partial_r\psi\big)^2\Big)+\frac{3}{2}\frac{\sin(2\psi)}{r^2}=0.
			$$
Due to the singularity at $r=0$ in the last term, we impose the condition $\psi(t,0)=0$ for all $t$. A direct calculation shows that the self-similar solution indeed satisfies this condition. Thus, it is natural to switch to the new independent variable 
			\begin{equation}
				u(t,r):=r^{-1}\psi(t,r). \label{rescaled soln}
			\end{equation}
		Doing so yields the equation
			\begin{equation}
				\Big(\partial_t^2u-\partial_r^2u-\frac{6}{r}\partial_ru\Big)-F\big(r u,r\partial_r u,r\partial_t u,r\big)=0 \label{rescaled semilinear sf skyrme eqn}
			\end{equation}
		where
			\begin{align}\label{Nonlin_Transformed_SF}
			\begin{split}
				F\big(r u,r\partial_r u,r\partial_t u,r\big)=&-\frac{1}{r}\cot(r u)\big((r\partial_t u)^2-(r\partial_r u)^2\big)-\frac{2}{r^2}\big(1-r u\cot(r u)\big)r\partial_r u
				\\
				&-\frac{\frac{3}{2}\sin(2r u)-2r u-(r u)^2\cot(r u)}{r^3}.
				\end{split}
			\end{align}
The solution \eqref{sf skyrme soln} transforms accordingly into
			$$
				u^T(t,r):=r^{-1}\psi^T_{SF}(t,r) = \frac{1}{T-t} \tilde U\left (\frac{r}{T-t} \right ),
			$$
for $\tilde U(\rho) := \rho^{-1} U(\rho)$. This variable transformation transforms the original equation into a semilinear radial wave equation in seven space dimensions (this approach has been used
frequently in the wave maps context). Furthermore, as long as $0\leq r u(t,r)<\pi$ for all $(t,r)$, with $r u(t,r)=0$ if and only if $r=0$, the nonlinearity is smooth. Throughout our analysis, we will show that, for sufficiently small perturbations of the blowup initial data, this property  is propagated throughout the flow. In the following, we denote the backward light cone in $(1+7)$-dimensions by 
\[ \mathfrak C_T := \{(t,x) \in [0,T) \times \R^7: 	 |x| \leq T-t \}. \]

The following result proves the nonlinear asymptotic stability of $u^T$ locally in a backward light cone modulo a small shift of the blowup time. In the statement of the theorem, we slightly abuse notation and identify radial functions with their radial representative.

\begin{theorem}\label{Th:Reform}
There are constants  $0 < \delta <1$, $c >1$ and $\omega >0$ such that the following holds. Let $(f,g) \in C^\infty_\text{rad}(\overline{\mathbb B^7_{2}})\times C^\infty_\text{rad}(\overline{\mathbb B^7_{2}})$ be real valued functions which satisfy 
\[
 \|(f,g)\|_{H^6(\mathbb B_{2}^7)\times H^5(\mathbb B_{2}^7)}\leq\frac{\delta}{c}.
\]
Then there exists a unique blowup time $T\in[1-\delta,1+\delta]$ depending Lipschitz continuously on $(f,g)$ and a unique solution  $u \in C^\infty_\text{rad}(\mathfrak C_T)$ of
Equation \eqref{rescaled semilinear sf skyrme eqn} satisfying  on $\overline {\mathbb B^7_T}$,
$$u(0,\cdot) = u^1(0,\cdot)+f,$$ $$\partial_t u(0,\cdot) = \partial_tu^1(0,\cdot)+g.$$  
Moreover, the solution has the decomposition
$$
u(t,r)  = \frac{1}{T-t} \left [\tilde U \bigg (\frac{r}{T-t} \bigg ) + \varphi \bigg (t,\frac{r}{T-t}  \bigg) \right ]
$$
with 
\begin{equation} \|(\varphi(t,\cdot), \partial_t \varphi(t,\cdot))  \|_{ H^5(\B^7) \times H^4(\B^7)} \lesssim  (T-t)^{\omega},\label{convergence}\end{equation}
for all $t \in [0,T)$. 
\end{theorem}

Some comments on the result are in order.

\begin{itemize}
\item Undoing the transformation \eqref{rescaled soln} yields a smooth solution  $\psi : \mc C_T \to \R$ of the original equation \eqref{sf skyrme eom} of the form 
\[ \psi(t,r) = \psi^T_{SF}(t,r) + \phi\left (t, \tfrac{r}{T-t}\right) \]
for every initial data sufficiently close to $\psi^1_{SF}$. Moreover, the perturbation decays to zero according to
\[ \||\cdot|^{-1} (\phi(t,\cdot), \partial_t \phi(t,\cdot))  \|_{ H^5(\B^7) \times H^4(\B^7)} \lesssim  (T-t)^{\omega}.\]
\item The regularity assumptions in Theorem \ref{The Main Result} ensure $L^{\infty}$-bounds for the perturbation and its time derivative, which allows us to define and control the nonlinearity.  In addition,  we assume smallness of the initial data in an even stronger topology, which we use to obtain Lipschitz-dependence on the blowup time via a fixed point argument, see Section \ref{Sec:Outline} for a more detailed explanation.\\
\item We have chosen to state the main result in the lowest possible dimension. In higher dimensions, the analogue of Theorem \ref{The Main Result} can be formulated provided that the spectral problem can be solved. However, we assume that the the generalization of the techniques implemented in this paper,  which are based on \cite{CDG17}, is straightforward for any given $d > 5$.   \\
\item We strongly conjecture that it is possible to prove blowup for the co-rotational Skyrme model in $d=5$ by using the profile  \eqref{Profile_U} together with the scaling properties of the full equation, see e.g. \cite{DW20}. We will motivate this conjecture in Remark \ref{conjecture}. This will be investigated in a forthcoming project.  \\
\end{itemize}

Before we proceed, we comment on a structural property of the full Skyrme model which we crucially exploit in the proof of 
Theorem \ref{The Main Result}.

\subsubsection{On the structure of the linearized Skyrme equation}\label{Lin_Skyrme}

The proof of Theorem \ref{Th:Reform} is based on the formulation of the evolution equation for small perturbations around the blowup initial data as an abstract Cauchy problem. The linearized problem is studied via semigroup methods including a detailed spectral analysis of the generator of the linearized flow in a highly non-self-adjoint setup. In order to translate spectral results into growth estimates for the corresponding evolution, we crucially exploit the following structural property of the Skyrme model, see also Section \ref{Sec:Outline}. By setting 
\[
				w(\psi)(t,r) :=\alpha r^{d-1}+\beta (d-1)r^{d-3}\sin^2(\psi(t,r)),
\]
		Equation \eqref{skyrme eom} can be written as
			\begin{align}
			\begin{split}
				\partial_t&(w(\psi) \partial_t\psi)-\partial_r ( w(\psi)\partial_r\psi)
				\\
				&+\frac{(d-1)r^{d-3}\sin(2\psi)}{2}\bigg(\alpha +\beta \Big(\frac{(d-2)\sin^2(\psi)}{r^2}+(\partial_r\psi)^2-(\partial_t\psi)^2\Big)\bigg)=0.
			\end{split} \label{Rewrite:Skyrme}
			\end{align}
Let $\Psi=\Psi(t,r)$ denote any sufficiently smooth solution of \eqref{Rewrite:Skyrme}. Linearizing around $\Psi$ yields a linear wave equation for the perturbation $\phi$ of the form
			\begin{equation}
				w(\Psi)\big(\partial_t^2\phi-\partial_r^2\phi\big)+\partial_tw(\Psi)\partial_t\phi-\partial_rw(\Psi)\partial_r\phi+V_{\Psi}(t,r)\phi=0 \label{linearized skyrme}
			\end{equation}
where $V_{\Psi}$ is some smooth potential. Upon setting
			$$
				\varphi:=\sqrt{w(\Psi)}\phi,
			$$
Equation \eqref{linearized skyrme} becomes
\begin{equation}\label{Lin:Standard_Form}
\sqrt{w(\Psi)}\big(\partial_t^2\varphi-\partial_r^2\varphi\big)+\tilde V_{\Psi}(t,r)\varphi=0
\end{equation}
for some smooth potential $\tilde V_{\Psi}$. In particular, \eqref{Lin:Standard_Form} does not contain first-order derivative terms. This property is remarkable since a \textit{single} change of variables cancels \textit{two} coefficients. That such a cancellation is possible is ensured by the form of the nonlinearity in Equation \eqref{Rewrite:Skyrme}. The original variable  $\phi$ can be recovered from the auxiliary variable $\varphi$ in any spacetime domain not containing zeros of $w(\Psi)$. When $\alpha\neq0$, a zero can only occur at $r=0$. If $\alpha=0$, like it is for the strong field Skyrme model, the invertibility of this transformation depends crucially on the background solution. In our case, we linearize around $\psi^T_{SF}$ which is strictly positive away from the origin and bounded away from $\pi$ within the backward light cone as long as $t < T$.

We note that this transformation is also used in the proof of the linear stability of the Skyrmion due to Creek, Donninger, Schlag, and Snelson \cite{CDSS16}. However, the Skyrmion is a static solution of Equation \eqref{Rewrite:Skyrme}. Thus, linearizing around the Skyrmion does not produce a $\partial_t\phi$-term in the analogue of Equation \eqref{linearized skyrme}. In this setting, removal of the $\partial_r\phi$-term mainly relies on properties of the background solution.  
For a time-dependent solution, like the self-similar solution $\psi^T_{SF}$, also the precise form of the full nonlinearity is essential to the successful removal of the corresponding first-order terms. 

\subsubsection{Discussion and related results}\label{Sec:related_results}

Closely related to Equations \eqref{skyrme eom} and \eqref{sf skyrme eom} are the wave maps equation \eqref{wave maps eom}, and the co-rotational, hyperbolic Yang-Mills equation given by
		\begin{equation}
			\partial_t^2\psi - \partial_r ^2\psi - \frac{d-3}{r} \partial_r \psi + \frac{(d-2)\psi(\psi+1)(\psi+2)}{r^2} = 0, \label{yang mills}
		\end{equation}
for $d \geq 3$. Both possess explicit self-similar solutions whose stability has been extensively studied over the past several years. In order to develop context around the present problem, we briefly summarize some of the results surrounding the stability of self-similar blowup. For a more general exposition on wave maps and the Yang-Mills equation, we refer the reader e.g. to Section 1.3 of \cite{G22} and Section 1.2 of \cite{G22a}, respectively.

The existence of self-similar solutions for Equation \eqref{wave maps eom} for $d=3$ was first proven by Shatah in \cite{S88} via variational techniques. Shortly thereafter, Turok and Spergel \cite{TS90} found what is believed to be an explicit form of Shatah's solution. More recently, the solution, as it is stated in Equation \eqref{wave maps soln}, was found by Biernat and Bizo\'n \cite{BB15}. The first nonlinear stability result within backward light cones for \eqref{wave maps soln} with $d=3$ is due to Donninger \cite{D11} based on the linearized results obtained by Aichelburg, Donninger, and the third author \cite{ADS12}. However, these results were conditional to a spectral assumption. The problem of spectral stability was then resolved by Costin, Donninger, and Xia \cite{CDX16} and Costin, Donninger, and Glogi\'c \cite{CDG17}. The extension of the stability result to all odd space dimensions $d\geq3$ is due to Chatzikaleas, Donninger, and Glogi\'c \cite{aCDG17}. Also, recently, for $d=4$, Donninger and Wallauch \cite{DW22} proved a nonlinear stability result at optimal regularity. Stable blowup for wave maps outside backward light cones has been established by Biernat, Donninger, and the third author \cite{BDS19} and Glogi\'c \cite{G22}. 

For the Yang-Mills equation, in dimensions $d=5,7,9$, the first construction of self-similar solutions for Equation \eqref{yang mills} is due to Cazenave, Shatah, and Tahvildar-Zadeh \cite{CST98}. Later, Bizo\'n \cite{B02} found this solution in closed form. The first rigorous proof of stable self-similar blowup within backward light cones is due to Donninger \cite{D14} and Costin, Donninger, Glogi\'c, and Huang \cite{CDGH16} for $d=5$; see also Biernat and Bizo\'n \cite{BB15} and Glogi\'c \cite{G22a} for the generalization to higher space dimensions. Stability outside the light cone has been analyzed by  Donninger and Ostermann \cite{DO21} as well as by Glogi\'c \cite{G22}.
	
Note that neither Equation \eqref{wave maps eom} nor \eqref{yang mills} possess quadratic or higher-order terms in the derivatives of the unknown. This is in stark contrast with Equations \eqref{skyrme eom} and \eqref{sf skyrme eom}. The present work therefor demonstrates how to deal with the additional difficulties that arise due to the presence of derivative nonlinearities in the Skyrme model. 

\subsection{Outline of the proof}\label{Sec:Outline}

We sketch the main steps in the proof of Theorem \ref{Th:Reform}.

\subsubsection*{Operator formulation in similarity coordinates}
Following the standard approach, we write the problem as a first-order system using similarity coordinates
\[ \t = - \log(T-t) + \log T, \quad  \xi = \frac{x}{T-t}  \]
for $(t,x)\in\mathfrak C_T$. This has the effect of transforming the stability of the self-similar solution $u^T$ into a more familiar nonlinear asymptotic stability problem for a static solution of a related evolution equation. The restriction of the independent variables to the backward light cone translates into $\tau \in [0, \infty)$, $\xi \in \overline{\mathbb B^7}$. The evolution of perturbations of $u^T$ is then governed by an operator equation of the form
\[
\partial_\t \Phi(\tau)= (\mathbf L_0+\mathbf L')\Phi(\tau) +\mathbf N(\Phi(\tau)) \label{first order eqn}
\]
for $\Phi(\tau) = (u_1,u_2)$ where $u_1$ and $u_2$ are suitable rescalings of $u$ and $\partial_tu$ in similarity coordinates. Here, $\mathbf L_0 = \mb L_W + \mb L_D$, where $\mb L_W$ represents the free wave part in similarity coordinates and, for 
$\mb u = (u_1,u_2)$, $\mb L_D \mb u = (0, -2 u_2)^{T}$ translates into a scale invariant damping term in physical coordinates. The operator $\mb N$ is the remaining nonlinearity.

\subsubsection*{The linearized flow} By exploiting the scaling properties of the problem, we prove exponential decay of the flow $(\mathbf{S}_0(\tau))_{\tau \geq 0}$ generated by $\mathbf L_0$ defined on a suitable domain $\mc D(\mb L_0) \subset \mc H : =H^5_\text{rad}(\mathbb B^7)\times H^4_\text{rad}(\mathbb B^7)$, see Proposition \ref{free evolution}. More precisely, we show that
\begin{align}
\label{Bound_S0_intro}
						\|\mathbf{S}_0(\tau)\mathbf u\|_\mathcal{H}\lesssim e^{-\frac{1}{2}\tau}\|\mathbf u\|_\mathcal H.
\end{align}
For this, we use a modified inner product analogous to \cite{GS21} which we generalize in order to control the flow in arbitrarily higher Sobolev norms (this is used to prove smoothness later on).  The existence of a semigroup  $(\mb S(\tau))_{\tau \geq 0 }$ generated by the linearized operator $\mb L = \mb L_0 + \mb L'$ follows from the boundedness of $\mb L'$. 

Explicitly, we have
\begin{align*}
				\mathbf{L}'\mathbf{u}(\xi):=
					\begin{pmatrix}
					0
					\\
					V_1(|\xi|)u_1(\xi)+ V_2(|\xi|)\big(|\xi|^2u_2(\xi)-\xi^j\partial_j u_1(\xi)\big)
				\end{pmatrix}
\end{align*}
for $\mb u  \in \mc H$ and smooth functions $V_1$ and $V_2$ to be specified later. The fact that $\mb L'$ contains a derivative prevents the operator from being compact (in fact, one can show that it is relatively compact with respect to $\mb L$). This is fundamentally different from previous problems to which the semigroup method has been applied. The structure of the perturbation causes major problems concerning the translation of spectral information into growth bounds for the corresponding semigroup. In fact, none of the soft arguments that have been used in previous works can be applied here. Of course, in view of the Gearhart-Pr\"uss-Greiner Theorem (see pg. 322, Theorem 1.11 of \cite{EN00}), constructing the resolvent of $\mb L$ and proving suitable uniform bounds for large imaginary parts would resolve the problem. However, this is a challenging and extremely technical endeavor. We avoid this by exploiting the structural property of the linearized Skyrme model, see Section \ref{Lin_Skyrme}, which we translate to our specific problem (such that its origin is not entirely obvious). In fact, we prove the existence of a bounded invertible operator $\mathbf\Gamma$ on $\mathcal H$ and a bounded operator $\mathbf V$ such that 
			$$
			\mathbf\Gamma \mathbf L \mathbf\Gamma^{-1} = 	\mathbf\Gamma(\mathbf L_0+\mathbf L')\mathbf\Gamma^{-1}=\mathbf L_0+\mathbf V =: \mb L_{\mb V} 
			$$
with
	$$
		\mathbf\Gamma\mathbf u(\xi):=
							\frac{\sqrt{w(U(|\xi|))}}{4|\xi|^2}\begin{pmatrix}
								1&0
								\\
								-\frac{1}{2}|\xi|^2V_2(|\xi|)&1
							\end{pmatrix}
							\begin{pmatrix}
								u_1(\xi)
								\\
								u_2(\xi)
							\end{pmatrix}
	$$
and $w$ as in Section \ref{Lin_Skyrme} with $\alpha=0$, $\beta=1$. Despite the apparent singularity at $\xi=0$, $\mathbf \Gamma$ is indeed invertible in $\mc H$ as will be shown in Section \ref{Semigroup theory}. The new operator $\mathbf V$ does not contain derivatives and turns out to be a \textit{compact} operator on $\mc H$. Thus, by extracting spectral information on $\mb L_{\mb V}$, see below, we can use merely the structure of $\mb L_{\mb V}$ together with the Biermann-Schwinger principle to get resolvent bounds for $\mb L_{\mb V}$, see Proposition \ref{Le:Spectum_Struct}, and thus bounds for the semigroup $(\mb S_{\mb V}(\tau))_{\tau \geq 0}$. The fact that $\mb S_{\mb V}(\tau) = \mb \Gamma \mb S(\tau) \mb \Gamma^{-1}$ for all $\tau \geq 0$  finally implies bounds on the linearized evolution.

\subsubsection*{Spectral analysis and growth bounds}
The spectral problem underlying the stability of self-similar solutions of nonlinear wave equations is notably difficult, since the highly non-self-adjoint nature largely prevents the application of standard methods. 

First note that $\sigma(\mb L) = \sigma(\mb L_{\mb V})$ by definition. It is easy to see that the time translation symmetry of the problem introduces the unstable eigenvalue $1 \in \sigma_p(\mb L)$. Also, the growth bound \eqref{Bound_S0_intro} in combination with the (relative) compactness of the potential immediately imply that 
\[ \sigma(\mb L) \cap \{ \lambda \in \C : \mathrm{Re} \lambda > -\tfrac{1}{2} \} \subset \sigma_p( \mb L ). \]
In the radial case, the eigenvalue equation $(\la-\mb L)\mb u=\mb 0$ can be reduced to a single second order ODE with singular coefficients for the first component of $\mb u$, see Lemma \ref{point spectrum and mode solutions}. A Frobenius analysis reveals that eigenfunctions have to be smooth inside the backward light cone including, in particular, the boundary. Following the by now standard approach developed in \cite{CDGH15,G18}, we prove that no smooth solutions exist for $ \lambda \in \C$ with $\mathrm{Re} \lambda \geq 0$ and $\lambda \neq 1$. In fact, $\lambda =1$ is an eigenvalue that is introduced by time-translation symmetry. We note that although the methods of \cite{CDGH15,G18} are systematic, they rely on the details of the underlying potential and their success is not guaranteed a priori. However, in our case, we are able to prove the existence of an $\omega_0 > 0$ such that 
\[
	\sigma(\mb L) = \sigma(\mb L_{\mb V}) \subseteq\{\lambda\in\mathbb{C}:\Re\lambda \leq -\omega_0 \}\cup\{1\}.
\]
Using the reasoning explained above, this translates into growth bounds for the linearized evolution. More precisely, we prove the existence of a  spectral projection $\mb P$ onto the eigenspace corresponding to $\lambda = 1$ such that  
\[ \|\mathbf S(\t)(1-\mathbf P)\mathbf u\|_\mathcal H\lesssim e^{-\omega\t}\|(1-\mathbf P)\mathbf u\|_\mathcal H \]
for some $\omega > 0$.

\subsubsection*{The nonlinear problem}

The nonlinear problem is treated via fixed point arguments relying on the integral formulation
\[ \Phi(\tau)=\mathbf S(\tau)\mb u+\int_0^\tau\mathbf S(\tau-s)\mathbf N\big(\Phi(s)\big)ds. \]
In order to ensure that the nonlinearity is defined and smooth, we have to guarantee that perturbations are pointwise small, which is granted by Sobolev embedding. Furthermore, the regularity imposed by $\mc H$ is sufficient to obtain local Lipschitz bounds for the nonlinearity by exploiting the algebra property of $H^k(\mathbb B^7)$ for $k \geq 4$. The rest of the proof follows standard arguments.

\begin{remark}[\textit{On the Blowup Conjecture for the $(5+1)$-dimensional Skyrme Model}] \label{conjecture}
	In similarity coordinates, it is possible to view the wave maps terms in Equation \eqref{skyrme eom} as lower-order compared to the strong field Skyrme terms nearby $\psi^T_{SF}$. Switching to the variable $u(t,r)=r^{-1}\psi(t,r)$, seeking a solution of the form $u(t,r)=u^T(t,r)+v(t,r)$ of the transformed equation and converting to similarity coordinates as described in Section \ref{Sec:Outline} yields an operator equation of the form
$$
\partial_\t \Phi(\tau)= (\mathbf L_0+\mathbf L')\Phi(\tau) +\mathbf N(\Phi(\tau))+T^2e^{-2\tau}\mathbf G_T(\Phi(\tau),\tau)
$$
where $\mathbf G_T$ contains the wave maps terms expanded around $\psi^T_{SF}$. By proving sufficient bounds on this term, it appears plausible that taking $T$ sufficiently small will yield solutions of Equation \eqref{skyrme eom} which remain close to $\psi^T_{SF}$ within $\mathfrak C_T$.
\end{remark}

\subsection{Notation and conventions}
Given $R>0$ and $n\in\mathbb N$, we denote by $\mathbb B_R^n:=\{x\in\mathbb R^n:|x|<R\}$ the open ball in $\mathbb R^n$ of radius $R$ centered at the origin. When $R=1$, we drop the subscript and simply write $\mathbb B^n$. By $\mathbb H$ we denote the open right-half plane in $\mathbb C$, i.e., $\mathbb H:=\{z\in\mathbb C:\Re z>0\}$. On a Hilbert space $\mathcal H$, we denote by $\mathcal B(\mathcal H)$ the space of bounded linear operators. For a closed operator $L$ on the Hilbert space $\mathcal H$ with domain $\mathcal D(L)$, we denote its resolvent set by $\rho(L)$ and by $R_L(\lambda):=(\lambda I- L)^{-1}$ the resolvent operator for $\lambda\in\rho( L)$. Furthermore, we denote by $\sigma( L):=\mathbb C\setminus\rho( L)$ the spectrum of $L$ and by $\sigma_p( L)$ its point spectrum. As we will only work with strongly continuous semigroups $\big(S(s)\big)_{s\geq0}$ of bounded operators on $\mathcal H$, we will instead refer to these more simply as semigroups on $\mathcal H$ whenever necessary. Given $x,y\geq0$, we say $x\lesssim y$ if there exists a constant $C>0$ such that $x\leq Cy$. Furthermore, we say that $x\simeq y$ if $x\lesssim y$ and $y\lesssim x$. If the constant $C$ depends on a parameter, say $k$, we will write $x\lesssim_ky$ when it is important to note the dependence on this parameter. 

\subsubsection{Function spaces}\label{Function Spaces}
For $R>0$, let
\[ C^\infty_\text{rad}(\overline{\mathbb B_R^7})=\{u\in C^\infty(\overline{\mathbb B_R^7}):u\text{ is radial}\}. \]
For $k\in\N$, we define the radial Sobolev space $H_\text{rad}^k(\mathbb B_R^7)$ as the completion of $C^\infty_\text{rad}(\overline{\mathbb B_R^7})$ under the standard Sobolev norm
\[ \|u\|_{H^k(\mathbb B_R^7)}^2:=\sum_{|\alpha|\leq k}\|\partial^\alpha u\|_{L^2(\mathbb B_R^7)}^2 \]
with $\alpha\in\mathbb N_0^7$ denoting a multi-index with $\partial^\alpha u=\partial_1^{\alpha_1}\dots\partial_d^{\alpha_d}u$ and $\partial_iu(x)=\partial_{x^i}u(x)$. In many places it will be convenient to work with radial representatives of functions in $C^\infty_\text{rad}(\overline{\mathbb B_R^7})$. That is, for any function $u\in C^\infty_\text{rad}(\overline{\mathbb B_R^7})$, there is a function $\hat u:[0,R]\to\mathbb C$ such that $u(x)=\hat u(|x|)$ for all $x\in\overline{\mathbb B^7}$. In fact, by Lemma 2.1 of \cite{G22a}, we have $\hat u\in C_e^\infty[0,R]$ where $C_e^\infty[0,R]$ denotes the space of `even' functions
\[
				C_e^\infty[0,R]:=\{u\in C^\infty[0,R]:u^{(2k+1)}(0)=0,\;k\in\N_0\}. \]
It will be convenient to also consider the space of `odd' functions, i.e., 
\[
				C_o^\infty[0,R]:=\{u\in C^\infty[0,R]:u^{(2k)}(0)=0,\;k\in\N_0\}. \]

\section{First-order formulation}\label{First-order formulation}
In this section, we perform some preliminary transformations, introduce similarity coordinates, and convert Equation \eqref{rescaled semilinear sf skyrme eqn} with initial data
	$$
		u(0,r)=u^1(0,r)+f(r),\quad\partial_tu(0,r)=\partial_tu^1(0,r)+g(r)
	$$ 
into a suitable abstract initial value problem for a first-order system.

For $T>0$, we define \textit{similarity coordinates} $(\t,\r)$ via the equation
			 $$
				(\t,\r):=\bigg(\log\Big(\tfrac{T}{T-t}\Big),\frac{r}{T-t}\bigg).
			$$
Restricting ourselves to the backward light cone $\mc C_T$ implies that $\rho \in [0,1]$ and $\tau \in [0,\infty)$.
By introducing rescaled dependent variables $\psi_1$ and $\psi_2$,
			$$
				\psi_1(\t,\r) =(T-t) u(t,r)|_{(t,r)=\big(t(\t,\r),r(\t,\r)\big)},\;\psi_2(\t,\r)=(T-t)^2\partial_t u(t,r)|_{(t,r)=\big(t(\t,\r),r(\t,\r)\big)},
			$$
Equation \eqref{rescaled semilinear sf skyrme eqn} becomes
			\begin{align*}
				\begin{pmatrix}
					\partial_\tau\psi_1
					\\
					\partial_\tau\psi_2
				\end{pmatrix}
				=&
				\begin{pmatrix}
					-\r\partial_\r-1&1
					\\
					\Delta_{\text{rad}}&-\r\partial_\r-2
				\end{pmatrix}
				\begin{pmatrix}
					\psi_1
					\\
					\psi_2
				\end{pmatrix}
				+
				\begin{pmatrix}
					0
					\\
					F(\r\psi_1,\r\partial_\r\psi_1,\r\psi_2,\r)
				\end{pmatrix}
			\end{align*}
		where $\Delta_{\text{rad}}=\partial_\r^2+\frac{6}{\r}\partial_\r$ denotes the seven-dimensional, radial Laplacian and $F$ is given by \eqref{Nonlin_Transformed_SF}. The linear portion of this equation is the seven-dimensional linear wave equation in our rescaled variables. The blowup solution transforms according to the equation
\[
				\begin{pmatrix}
					(T-t)u^T(t,r)
					\\
					(T-t)^2 \partial_t u^T(t,r)
				\end{pmatrix}\Bigg|_{(t=t(\tau,\rho),r=r(\tau,\rho))}
				=
				\begin{pmatrix}
					\tilde U(\r)
					\\
					U'(\r)
				\end{pmatrix}=:\begin{pmatrix} U_1(\rho) \\ U_2(\rho) \end{pmatrix}. \]
In particular, observe that the blowup solution corresponding to blowup time $T$ is static in these coordinates.  
Inserting the ansatz 
\[ 		
	\begin{pmatrix}
	\psi_1(\tau,\rho)
	\\
	\psi_2(\tau,\rho)
	\end{pmatrix} =	
		\begin{pmatrix}
					U_1(\rho)
					\\
					U_2(\rho)
				\end{pmatrix} 	+ \begin{pmatrix}
					\varphi_1(\tau,\rho)
					\\
					\varphi_2(\tau,\rho)
				\end{pmatrix} \] 
yields 
			\begin{align}
			\begin{split}
				\begin{pmatrix}
					\partial_\tau\varphi_1
					\\
					\partial_\tau\varphi_2
				\end{pmatrix}
				=&
									\begin{pmatrix}
					-\r\partial_\r-1&1
					\\
					\Delta_{\text{rad}}&-\r\partial_\r-2
				\end{pmatrix}
				\begin{pmatrix}
					\varphi_1
					\\
					\varphi_2
				\end{pmatrix}
				+
				\begin{pmatrix}
					0
					\\
					V_1(\r)\varphi_1+\mathring V_1(\r)\partial_\r\varphi_1+\mathring V_2(\r)\varphi_2
				\end{pmatrix}
				\\
				&+
				\begin{pmatrix}
					0
					\\
					N(\r\varphi_1,\r\partial_\r\varphi_1,\r\varphi_2,\r)
				\end{pmatrix} \label{expanded first order eqn}
			\end{split}
			\end{align}
		where $V_1, \mathring V_1, \mathring V_2\in C_e^\infty[0,1]$ are given explicitly by
\begin{align}\label{Def_V1}
V_1(\r) & =\r\partial_2 F(\r,\r U_1,\r \partial_\rho U_1 ,\r U_2)=-\frac{5\big(21\rho^6-375\rho^4+1455\rho^2-2125\big)}{\big(5+3\rho^2\big)^2\big(5-\rho^2\big)^2},
\end{align}
$$
\mathring V_1(\r)  =\r\partial_3F(\r,\r U_1,\r \partial_\rho U_1 ,\r U_2)=\frac{2\r\big(3\rho^2-35\big)}{\big(5+3\rho^2\big)\big(5-\rho^2\big)},
$$
and 
$$
\mathring V_2(\r)  =\r\partial_4 F(\r,\r U_1,\r \partial_\rho U_1 ,\r U_2)=-\frac{50\big(1-\r^2\big)}{\big(5+3\rho^2\big)\big(5-\rho^2\big)},
$$
	and $N$, the nonlinear remainder, is given by  
			\begin{align}
			\begin{split}
				N(\r\varphi_1,\r\partial_\r\varphi_1,\r\varphi_2,\r) & = 
F(\r,\r U_1+\r\varphi_1,\r \partial_\rho U_1  +\r\partial_\r\varphi_1,\r U_2+\r\varphi_2)\\
& -F(\r,\r U_1,\r \partial_\rho U_1 ,\r U_2) -V(\rho) \varphi_1-\mathring V_1(\r) \partial_\r\varphi_1 -\mathring V_2(\r)\varphi_2. \label{scalar nonlinear remainder}
			\end{split}
			\end{align}
In order to treat the second term in Equation \eqref{expanded first order eqn} perturbatively, we use the identity
\[ \mathring V_2(\r)= 2 -\r \mathring V_1(\r) \]
to rewrite Equation \eqref{expanded first order eqn} as
			\begin{align}
			\begin{split}
				\begin{pmatrix}
					\partial_\tau\varphi_1
					\\
					\partial_\tau\varphi_2
				\end{pmatrix}
				=&
													\begin{pmatrix}
					-\r\partial_\r-1&1
					\\
					\Delta_{\text{rad}}&-\r\partial_\r-4
				\end{pmatrix}
				\begin{pmatrix}
					\varphi_1
					\\
					\varphi_2
				\end{pmatrix}
				+
				\begin{pmatrix}
					0
					\\
					V_1(\r)\varphi_1+ V_2(\rho)\rho\big(\rho \varphi_2-\partial_\r\varphi_1\big)
				\end{pmatrix}
				\\
				&+
				\begin{pmatrix}
					0
					\\
					N(\r\varphi_1,\r\partial_\r\varphi_1,\r\varphi_2,\r)
				\end{pmatrix} \label{modified expanded first order eqn}
			\end{split}
			\end{align}
		where $V_2(\rho) := - \rho^{-1} \mathring V_1(\rho)$ is given explicitly by
\begin{align}\label{V2_Def}
 V_2(\rho)=-\frac{2\big(3\rho^2-35\big)}{\big(5+3\rho^2\big)\big(5-\rho^2\big)}. 
\end{align}
Furthermore, by a direct calculation, one sees that the initial data becomes
\begin{align}\label{Initia_Data_Trans}
		\begin{pmatrix}
	    \varphi_1(0,\cdot)
					\\
	   \varphi_2(0,\cdot)
				\end{pmatrix} = \begin{pmatrix}
	    T U_1(T \cdot) - U_1(\cdot)  + T(\cdot)^{-1}f(T \cdot) 
					\\
	   T^2 U_2(T \cdot) - U_2(\cdot) + T^2(\cdot)^{-1}g(T \cdot) 
				\end{pmatrix} 
\end{align}

In the following, we treat  \eqref{modified expanded first order eqn} and \eqref{Initia_Data_Trans} as an abstract initial value problem on a Sobolev space of radial functions. More precisely, we define
\[ \mc H^k := 	H_{\text{rad}}^k(\mathbb B^7)	 \times H_{\text{rad}}^{k-1}(\mathbb B^7) \]
which comes equipped with the norm 
\[ \|\mathbf u\|_{\mathcal H^k}^2:=\|u_1\|_{H^k(\mathbb B^7)}^2+\|u_2\|_{H^{{k-1}}(\mathbb B^7)}^2 \]
for $\mathbf u=(u_1,u_2)$ and the dense subset
$C^\infty_\text{rad}(\overline{\mathbb B^7}) \times C^\infty_\text{rad}(\overline{\mathbb B^7})$. Central to our analysis is the space $\mc H^5$ which we will more simply denote as $\mc H$.
		
\section{The linear time evolution} \label{The linear time evolution}

For $\xi\in\mathbb B^7$ and $\mathbf u=(u_1,u_2)\in C^\infty_\text{rad}(\overline{\mathbb B^7}) \times C^\infty_\text{rad}(\overline{\mathbb B^7})$, we define
\[
				\tilde{\mathbf{L}}_0\mathbf u(\xi):=
					\begin{pmatrix}
					-\xi^j \partial_j-1&1
					\\
					\Delta &- \xi^j \partial_j-4
				\end{pmatrix}
				\begin{pmatrix}
					u_1(\xi)
					\\
					u_2(\xi)
				\end{pmatrix}				
\]
where $\partial_j=\partial_{\xi^j}$. Equipped with the domain $\mathcal D(\tilde{\mathbf L}_0):=C^\infty_\text{rad}(\overline{\mathbb B^7})\times C^\infty_\text{rad}(\overline{\mathbb B^7})$, the unbounded operator $\big(\tilde{\mathbf L}_0,\mathcal D(\tilde{\mathbf L}_0)\big)$ is densely-defined on $\mathcal H$. Writing $\tilde{\mathbf{L}}_0\mathbf u$ in terms of radial representatives gives exactly the first term on the right-hand side of Equation \eqref{modified expanded first order eqn}.  We note that $\tilde{\mb L}_0$ does not describe the free wave evolution, but corresponds to a damped wave equation with a scale invariant damping term in physical coordinates.
		
		Furthermore, on $C^\infty_\text{rad}(\overline{\mathbb B^7}) \times C^\infty_\text{rad}(\overline{\mathbb B^7})$ we define
$$
				\mathbf{L}'\mathbf{u}(\xi):=
					\begin{pmatrix}
					0
					\\
					V_1(|\xi|)u_1(\xi)+ V_2(|\xi|)\big(|\xi|^2u_2(\xi)-\xi^j\partial_j u_1(\xi)\big)
				\end{pmatrix},
$$
with $V_1, V_2 \in C_e^\infty[0,1]$ defined in \eqref{Def_V1} and \eqref{V2_Def} respectively. Note that  $\mathbf{L}'$  extends to a bounded operator on  $\mathcal H$ which we again denote by $\mathbf L'$.

\subsection{Semigroup theory} \label{Semigroup theory}

			\begin{proposition}\label{free evolution}
				The operator $\big(\tilde{\mathbf{L}}_0,\mathcal D(\tilde{\mathbf{L}}_0)\big)$ is closable in $\mathcal H$ and its closure, denoted by $\big(\mathbf{L}_0,\mathcal D(\mathbf{L}_0)\big)$, is the generator of a semigroup on $\mathcal H$, $(\mathbf{S}_0(\tau))_{\tau\geq0}$, satisfying the estimate
					$$
						\|\mathbf{S}_0(\tau)\mathbf u\|_\mathcal{H}\leq Me^{-\frac{1}{2}\tau}\|\mathbf u\|_\mathcal H
					$$
				for all $\tau\geq0$, $\mathbf u \in\mathcal H$, and for some constant $M \geq 1$.
			\end{proposition}

The proof of the growth bound necessitates the use of an equivalent inner product on $\mc H$ along the lines of \cite{GS21}. We defer the proof to Appendix \ref{The Free Wave Evolution}.

In view of Proposition \ref{free evolution} and the boundedness of $\mathbf L'$, we infer closedness of the operator 
\[ \mathbf L := \mathbf L_0+\mathbf L' \]
with domain $\mathcal D(\mathbf L):=\mathcal D(\mathbf L_0) \subset \mc H$. The following statement is a simple consequence of the Bounded Perturbation Theorem (see \cite{EN00}, p. 158, Theorem 1.3).
			
\begin{proposition} \label{perturbed evolution}
				The operator $(\mathbf L, \mathcal D(\mathbf L))$  is the generator of a semigroup on $\mathcal H$, $(\mathbf{S}(\tau))_{\tau\geq0}$, satisfying the estimate
					\begin{equation}
						\|\mathbf{S}(\tau)\mathbf u\|_\mathcal{H}\leq Me^{(-\frac{1}{2}+M\|\mathbf L'\|)\tau}\|\mathbf u \|_\mathcal H\label{bpt bound}
					\end{equation}
				for all $\tau\geq0$, $\mathbf  u \in\mathcal H$, and $M \geq 1$ as in Proposition \ref{free evolution}.
			\end{proposition}

The bound \eqref{bpt bound} is too weak to control the linear evolution since the norm of $\mathbf L'$ and $M$ are large. In fact, decay is not true in general due to the presence of an eigenvalue of $\mathbf L$ with positive real part. Thus, our aim for the rest of this section is to show that, despite this anticipated instability, decay of the semigroup can be obtained on a suitable subspace. For this, we require information on the spectrum of $\mb L$ together with a spectral mapping property. As described in Section \ref{Sec:Outline}, the latter is difficult to obtain in general for perturbations containing derivatives. However, we exploit the following special structural property of the linearized equation to reduce matters to a compactly perturbed problem.

			\begin{proposition}\label{compactification}
				There exists an invertible operator $\mathbf\Gamma\in\mathcal B(\mathcal H)$ and a compact operator $\mathbf V$ on $\mathcal H$ such that
					\begin{equation}
						\mathbf\Gamma(\mathbf L_0+\mathbf L')\mathbf\Gamma^{-1}\mathbf u=(\mathbf L_0+\mathbf V)\mathbf u \label{compactification of L}
					\end{equation}
				for all $\mathbf u\in\mathcal D(\mathbf L_0)$.	
			\end{proposition}
			\begin{proof}
				For $\mathbf u=(u_1,u_2)\in\mathcal H$ and $\xi\in\overline{\mathbb B^7}$, consider the expression
					$$
						\mathbf\Gamma\mathbf u(\xi):=
							\frac{\sqrt{5-|\xi|^2}}{5+3|\xi|^2}\begin{pmatrix}
								1&0
								\\
								-\frac{1}{2}|\xi|^2V_2(|\xi|)&1
							\end{pmatrix}
							\begin{pmatrix}
								u_1(\xi)
								\\
								u_2(\xi)
							\end{pmatrix}.
					$$
				Observe that this pointwise definition makes sense via the embedding $\mc H\hookrightarrow C^1(\overline{\mathbb B^7})\times C(\overline{\mathbb B^7})$. A direct calculation shows that $\mathbf\Gamma\in\mathcal B(\mathcal H)$ and that it is invertible with inverse given by
					$$
						\mathbf\Gamma^{-1}\mathbf u(\xi)=
							\frac{5+3|\xi|^2}{\sqrt{5-|\xi|^2}}\begin{pmatrix}
								1&0
								\\
								\frac{1}{2}|\xi|^2V_2(|\xi|)&1
							\end{pmatrix}
							\begin{pmatrix}
								u_1(\xi)
								\\
								u_2(\xi)
							\end{pmatrix}.
					$$	
				For $\mathbf u\in C^\infty_\text{rad}(\overline{\mathbb B^7}) \times C^\infty_\text{rad}(\overline{\mathbb B^7})$, a direct calculation verifies that the equation
					\begin{equation}
						\mathbf\Gamma(\tilde{\mathbf L}_0+\mathbf L')\mathbf\Gamma^{-1}\mathbf u=(\tilde{\mathbf L}_0+\mathbf V)\mathbf u \label{open compactification of L}
					\end{equation}
				holds with 
					$$
						\mathbf V\mathbf u(\xi):=
							\begin{pmatrix}
								0
								\\
								\tilde{V}(|\xi|)u_1(\xi)
							\end{pmatrix},\;\;\;
						\tilde  V(\r)=-\frac{2(9\rho^4+102\rho^2-335)}{(5+3\rho^2)^2}.
					$$
				Note that $\tilde  V\in C_e^\infty[0,1]$ and that this implies $\mathbf V\in\mathcal B(\mathcal H)$. Compactness of $\mathbf V$ then follows from compactness of the embedding $H^5(\mathbb B^7)\hookrightarrow H^4(\mathbb B^7)$. By Proposition \ref{free evolution} and boundedness of $\mathbf V$, we infer closedness of the operator $\mathbf L_\mathbf V:=\mathbf L_0+\mathbf V$ with domain $\mathcal D(\mathbf L_\mathbf V):=\mathcal D(\mathbf L_0)\subset\mathcal H$.
					
				We now show that $\mathbf u\in\mathcal D(\mathbf L_0)$ if and only if $\mathbf\Gamma^{-1}\mathbf u\in\mathcal D(\mathbf L_0)$. As a consequence, Equation \eqref{compactification of L} follows. To that end, suppose $\mathbf u\in\mathcal D(\mathbf L_0)$, i.e., there is a sequence $\{\mathbf u_n\}_{n\in\mathbb N}\subset\mathcal D(\tilde{\mathbf L}_0)$ for which $\mathbf u_n\to \mathbf u$ and $\tilde{\mathbf L}_0\mathbf u_n \to  \mathbf L_0\mathbf u$ in $\mc H$. In particular, $\{\tilde{\mathbf L}_0\mathbf u_n\}$ is Cauchy in $\mathcal H$ and, since $\mathbf V\in\mathcal B(\mathcal H)$, we have also that $\{(\tilde{\mathbf L}_ 0+\mathbf V) \mb u_n\}_{n\in\mathbb N}$ is Cauchy in $\mathcal H$ with its limit defining the expression $\mathbf L_\mathbf V\mathbf u$. By a direct calculation, we see that $\{\mathbf\Gamma^{-1}\mathbf u_n\}_{n\in\mathbb N}\subset C^\infty_ \text{rad}(\overline{\mathbb B^7})\times C^\infty_\text{rad}(\overline{\mathbb B^7})$ and, by continuity, $\mathbf\Gamma^{-1}\mathbf u_n\to\mathbf\Gamma^{-1}\mathbf u$ in $\mathcal H$. Rearranging Equation \eqref{open compactification of L} yields
					\[
						\tilde{\mathbf L}\mathbf\Gamma^{-1}\mathbf u_n=\mathbf\Gamma^{-1}\tilde{\mathbf L}_\mathbf V\mathbf u_n\to\mathbf\Gamma^{-1}\mathbf L_\mathbf V\mathbf u
					\]
				with the limit being taken in $\mathcal H$. Thus, the sequence $\{\tilde{\mathbf L}\mathbf\Gamma^{-1}\mathbf u_n\}_{n\in\mathbb N}$ converges in $\mathcal H$ and since $\mb L$ is closed, we infer that $\mathbf\Gamma^{-1}\mathbf u\in\mathcal D(\mathbf L) = \mathcal D(\mathbf L_0)$. The converse is established analogously.
			\end{proof}
			
By Proposition \ref{free evolution} and another application of the Bounded Perturbation Theorem, we infer  that $\mathbf L_\mathbf V$ generates a semigroup on $\mc H$, which we denote by $\big(\mathbf S_\mathbf V(\t)\big)_{\t\geq0}$, and satisfies
			$$
				\|\mathbf{S}_\mathbf V(\tau)\mathbf u \|_\mathcal{H}\leq M'e^{(-\frac{1}{2}+M'\|\mathbf V\|)\tau}\|\mathbf u\|_\mathcal H\label{Vbpt bound}
			$$
		for all $\tau\geq0$, $\mathbf u \in\mathcal H$, and some $M' \geq 1$. As immediate corollaries of Proposition \ref{compactification}, we obtain the following two crucial results.
		
			\begin{corollary} \label{Similar Semigroup}
We have
\begin{equation}
		\mathbf\Gamma\mathbf S(\t)\mathbf\Gamma^{-1} = \mathbf S_\mathbf V(\t) \label{similar semigroup}
					\end{equation}
for all $ \tau \geq 0$.
			\end{corollary}
			\begin{proof}
				Observe that for $\mathbf u\in\mathcal D(\mathbf L)$ and $\tau > 0$,
					\begin{align*}
						\frac{\mathbf\Gamma\mathbf S(\t)\mathbf\Gamma^{-1}\mathbf u-\mathbf u}{\t}=\frac{\mathbf\Gamma(\mathbf S(\t)\mathbf-1)\mathbf\Gamma^{-1}\mathbf u}{\t}\to_{\t\to0^+}\mathbf\Gamma\mathbf L\mathbf\Gamma^{-1}\mathbf u=\mathbf L_\mathbf V\mathbf u.
					\end{align*}
				This shows that $\mathbf L_\mathbf V$ generates the semigroup $\big(\mathbf\Gamma\mathbf S(\t)\mathbf\Gamma^{-1}\big)_{\t\geq0}$. However, $\mathbf L_\mathbf V$ generates the semigroup $\big(\mathbf S_\mathbf V(\t)\big)_{\t\geq0}$. By Theorem 1.4, p. 51 of \cite{EN00}, semigroups are uniquely determined by their generator. Thus, Equation \eqref{similar semigroup} holds.
			\end{proof}
			
			\begin{corollary} \label{isospectral}
				We have  $\sigma(\mathbf L)=\sigma(\mathbf L_\mathbf V)$. In particular, if $\la\in\sigma_p(\mathbf L)$ with eigenfunction $\mathbf f$, then $\la\in\sigma_p(\mathbf L_\mathbf V)$ with eigenfunction $\mathbf\Gamma\mathbf f$. Conversely, if $\la\in\sigma_p(\mathbf L_\mathbf V)$ with eigenfunction $\mathbf f$, then $\la\in\sigma_p(\mathbf L)$ with eigenfunction $\mathbf\Gamma^{-1}\mathbf f$. 
			\end{corollary}
			\begin{proof}
				Equation \eqref{compactification of L} implies the first claim. Now suppose that $\la\in\sigma_p(\mathbf L)$ and that $\mathbf f\in\mathcal D(\mathbf L)\setminus\{\mathbf 0\}$ is any associated eigenfunction. Then, again by Equation \eqref{compactification of L}
					$$
						(\la-\mathbf L_\mathbf V)\mathbf\Gamma\mathbf f=\mathbf\Gamma(\la-\mathbf L)\mathbf f=\mathbf 0.
					$$ 
				Since $\mathbf\Gamma\mathbf f\in\mathcal D(\mathbf L_\mathbf V)\setminus\{\mathbf 0\}$, it follows that $\la\in\sigma_p(\mathbf L_\mathbf V)$ and $\mathbf\Gamma\mathbf f$ is an eigenfunction. The converse follows mutatis mutandis.
			\end{proof}

\subsection{Estimates for the time evolution described by $\mb L_{\mb V}$}\label{Estimates for the time evolution described by L_V}

\subsubsection{Spectral analysis}\label{Spectral analysis}

For $\omega  \in \R$, we define
\[ \overline{\mathbb H}_{\omega} := \{ \lambda \in \C: \mathrm{Re} \lambda \geq \omega\}\]
and write $ \overline{\mathbb H}:= \overline{\mathbb H}_0$. We have the following characterization.

\begin{proposition}\label{Le:Spectum_Struct}
Let $\lambda \in \sigma(\mathbf L_\mathbf V)\cap\overline{\mathbb H}_{-\frac14}$. Then $\lambda$ is an isolated eigenvalue of finite algebraic multiplicity. Moreover, there exist $C, K > 0$ such that 
\[ \| \mb R_{\mb L_{\mb V}}(\lambda) \mb u \|_{\mc H} \leq  C\| \mb u \|_{\mc H} \]
for all $\lambda \in \overline{\mathbb H}_{-\frac14}$ with $|\lambda| \geq K$ and all $\mb u \in \mc H$. In particular, the 
set $\sigma_p (\mathbf L_\mathbf V)\cap \overline{\mathbb H}_{-\frac14}$ is finite. 
\end{proposition}

\begin{proof}
The proof is based on standard arguments using the compactness of $\mb V$ along with the identity
\begin{align}\label{Decomp}
 \lambda - \mb L_{\mb V} = (1 - \mb V \mb R_{\mb L_0}(\lambda))(\lambda - \mb L_0)
\end{align}
and the properties of $\mb L_0$. The first part of the statement is an immediate consequence of the analytic Fredholm Theorem. The resolvent estimates, again based on \eqref{Decomp}, are proved by a Neumann series argument using the fact that 
\begin{equation} \|\mb V \mb R_{\mb L_0}(\lambda) \mb f  \|_{\mc H} \lesssim \| [\mb R_{\mb L_0}(\lambda) \mb f]_{1} \|_{H^{4}(\B^7)}\label{useful inequality} \end{equation}
for all $\mb f  \in \mc H$. More precisely, the identity $(\lambda - \mb L_0)\mb R_{\mb L_0}(\lambda) \mb f = \mb f$ implies
\begin{equation} \xi^{i} \partial_i [\mb R_{\mb L_0}(\lambda) \mb f]_{1}(\xi) + (\lambda+1) [\mb R_{\mb L_0}(\lambda) \mb f]_{1}(\xi) - [\mb R_{\mb L_0}(\lambda) \mb f]_{2}(\xi) = f_1(\xi).\label{comp12 relation} \end{equation}
Using the uniform boundedness of $\mb R_{\mb L_0}(\lambda)$, which is a consequence of Proposition \ref{free evolution} under the above assumptions on $\lambda$, we infer that 
\begin{align*}
\| [\mb R_{\mb L_0}(\lambda) \mb f]_{1} \|_{H^{4}(\B^7)} &  \lesssim \frac{1}{|\lambda+1|} \left (\| [\mb R_{\mb L_0}(\lambda) \mb f]_{1} \|_{H^{5}(\B^7)}  + \| [\mb R_{\mb L_0}(\lambda) \mb f]_{2 } \|_{H^{4}(\B^7)} + \|f_1 \|_{H^{4}(\B^7)}  \right )  \\
& \lesssim \frac{1}{|\lambda+1|} \| \mb f\|_{\mc H}.
\end{align*}
Hence,  $\|\mb V \mb R_{\mb L_0}(\lambda) \mb f \|_{\mc H} \leq \frac{1}{2} \|\mb f \|_{\mc H}$
 for all $\lambda \in \overline{\mathbb H}_{-\frac14}$ with $|\lambda| > K$ and $K > 0$ sufficiently large. Now, 
 $\mb R_{\mb L_{\mb V}}(\lambda) =  \mb R_{\mb L_0}(\lambda) \sum_{k=0}^{\infty} [\mb V \mb R_{\mb L_0}(\lambda)]^k$, which implies  the claimed estimate.
\end{proof}

\begin{remark}
In view of Proposition \ref{compactification} we immediately obtain uniform bounds for the resolvent $\mb R_{\mb L}(\lambda)$ of the original operator $\mb L = \mb L_0 + \mb L'$, which is crucial in proving bounds on the linear time evolution using spectral properties of its generator. We emphasize that it is not obvious how to obtain such bounds without exploiting the reduction provided by Proposition \ref{compactification}. Crucial to the previous argument are \eqref{useful inequality} and \eqref{comp12 relation}. Since the first component of the resolvent is measured in $H^4(\mathbb B^7)$ in \eqref{useful inequality} and this is the level of regularity for which the second component of the resolvent is controlled, we can use Equation \eqref{comp12 relation} to gain the desired decay in $\la$. In contrast, observe that
	$$
		\|\mb L' \mb R_{\mb L_0}(\lambda) \mb f  \|_{\mc H} \lesssim \| [\mb R_{\mb L_0}(\lambda) \mb f]_{1} \|_{H^{5}(\B^7)}+ \| [\mb R_{\mb L_0}(\lambda) \mb f]_{2} \|_{H^{4}(\B^7)}.
	$$
Clearly, the first component of the resolvent is measured in $H^{5}(\B^7)$ here and, as a consequence, Equation \eqref{comp12 relation} is of no use.
\end{remark}

We proceed by analyzing the spectrum of $\mb L_{\mb V}$ (equivalently of $\mb L$). The following lemma shows that the question of spectral stability can be reduced to an ODE problem. 
		
		\begin{lemma} \label{point spectrum and mode solutions}
			Let $\la\in\sigma(\mathbf L_\mathbf V)\cap\overline{\mathbb H}$. Then there exists a nonzero $f\in C^\infty[0,1]$ such that
				\begin{align}
				\begin{split}
					-(1-\rho^2)f''(\rho)-\Big(\tfrac{6}{\rho}-2(\lambda+3)\rho\Big)f'(\rho)+\Big((\lambda+1)(\lambda+4)-\tilde V(\rho)\Big)f(\rho)=0. \label{spectral equation}
				\end{split}
				\end{align}
		\end{lemma}
		\begin{proof}

Suppose $\la\in\sigma(\mathbf L_\mathbf V)\cap\overline{\mathbb H}$. By Proposition \ref{Le:Spectum_Struct}, $\lambda$ is an eigenvalue. Thus, there exists $\mathbf f=(f_{1},f_{2})\in\mathcal D(\mathbf L_\mathbf V)\setminus\{\mathbf 0\}$ such that $(\la-\mathbf L)\mathbf f_\la=\mathbf 0$. This implies that the radial representatives of $f_{1}$ and $f_{2}$, denoted by $\hat f_1$ and $\hat f_2$ respectively, solve
				$$
					\begin{cases}
						\rho\hat f_{1}'(\rho)+(\lambda+1)\hat f_{1}(\rho)-\hat f_{2}(\rho)=0
						\\
						-\hat f_{1}''(\rho)-\frac{6}{\rho}\hat f_{1}'(\rho)+\rho\hat f_{2}'(\rho)+(\lambda+4)\hat f_{2}(\rho)-\tilde V(\rho)\hat f_{1}(\rho)-\hat f_{1}'(\rho)=0
					\end{cases}
				$$
			on the interval $(0,1)$. Using the first equation to solve for $\hat f_{2}$ in terms of $\hat f_{1}$ and its derivative, we find that $\hat f_{1}$ solves Equation \eqref{spectral equation} on the interval $(0,1)$. Since the coefficients are smooth on $(0,1)$, we have $\hat f_{1}\in C^\infty(0,1)$. To see that we have smoothness up to the endpoints, we perform a Frobenius analysis of Equation \eqref{spectral equation}. We begin at the regular singular point $\r=1$ where the Frobenius indices are $\{0,1-\la\}$. There are three cases to consider:\\
			
\textit{Case 1} ($\la=0$ or $\la=1$): In this case, Equation \eqref{spectral equation} has a fundamental system of the form
				\[
					v^{+}_1(\r;\la)=(1-\r)^{1-\la}h_1(\r;\la),\;\;\;v_1^{-}(\r;\la)=h_2(\r;\la)+c\log(1-\r)v^{+}_1(\r;\la),
				\]
				where $c\in\mathbb C$,  and $h_1(\cdot;\la), h_2(\cdot;\la)$ are analytic in a neighborhood of $\r=1$ with $h_1(1;\la) = h_2(1;\la)=1$. Since $f_{1}\in H^5_\text{rad}(\mathbb B^7)$, either $c=0$ or $f_{1}=c_1 v^{+}_1(\cdot;\la)$ for some $c_1\in\mathbb C$. In either case, $f_{1}\in C^\infty(0,1]$.\\
			
			\textit{Case 2} ($\la-1\in\mathbb N_0$ and $\Re\la>1$): Similar to Case 1, Equation \eqref{spectral equation} has a fundamental system of the form
\[
					v^{+}_1(\r;\la)=h_1(\r;\la),\;\;\;v_1^{-}(\r;\la)=(1-\r)^{1-\la}h_2(\r;\la)+c\log(1-\r)v_1^{+}(\r;\la)
\]
			where $c\in\mathbb C$ and $h_1(\cdot;\la),h_2(\cdot;\la)$ analytic in a neighborhood of $\r=1$ with $h_1(1;\la) = h_2(1;\la)=1$. However, due to $1-\la\leq-1$ in this case, we immediately conclude $f_{\la,1}=c_1 v_1^{+}(\cdot;\la)$ for some $c_1\in\mathbb C$ which implies $f_{1}\in C^\infty(0,1]$.\\
			
			\textit{Case 3} ($1-\la\not\in\mathbb N_0$): In this case, Equation \eqref{spectral equation} admits a fundamental system of the form
				$$
					v_1^{+}(\r;\la)= h_1(\r;\la),\;\;\;v_1^{-}(\r;\la)=(1-\r)^{1-\la}h _2(\r;\la)
				$$
			with $h_1(\cdot;\la),h_2(\cdot;\la)$  analytic in a neighborhood of $\r=1$ and $h_1(1;\la) = h_2(1;\la)=1$. We infer that  $f_{1}=c_1 h_1(\cdot;\la)$ for some $c_1\in\mathbb C$ which implies $f_{1}\in C^\infty(0,1]$.\\

We conclude by proving smoothness at $\r=0$. Observe that at $\rho =0$ the Frobenius indices are $\{0,-5\}$ and thus Equation \eqref{spectral equation} admits a fundamental system
				$$
					v^+_1(\r;\la)=h_1(\r;\la),\;\;\;v^-_2(\r;\la)=\r^{-5}h_2(\r;\la)+c\log(\r)h_1(\r;\la)
				$$
			where $c\in\mathbb C$, $h_1(\cdot;\la),h_2(\cdot;\la)$ are analytic in a neighborhood of $\r=0$, and $h_1(0;\la)=h_2(0;\la)=1$. Again, since $f_{\la,1}\in H^5_\text{rad}(\mathbb B^7)$, we must have $f_{1}(\cdot)=c_1v^+_1(\cdot;\la)$ for some $c_1\in\mathbb C$ which implies $f_{1}\in C^\infty[0,1]$.
		\end{proof}
		
Hence, in order to characterize the spectrum of $\mb L_{\mb V}$ in the right half plane, we define the set
			\[
				\Sigma:=\{\lambda\in\mathbb C:\Re\lambda\geq0\text{ and }\exists f(\cdot;\lambda) \in C^\infty[0,1]\text{ solving Equation }\eqref{spectral equation}\text{ on }(0,1)\}.
			\]
			\begin{proposition}\label{mode stability}
				We have
					\[
						\Sigma=\{1\}
					\]
				with unique, up to a constant multiple, solution $f(\rho;1)=(5+3\rho^2)^{-2}$.
			\end{proposition}
			\begin{proof}
				By direct computation, one sees that $f(\rho;1)$ solves Equation \eqref{spectral equation} with $\lambda=1$. To see the reverse inclusion we show that Equation \eqref{spectral equation} does not admit solutions which are smooth on $[0,1]$ for $\lambda \in \C$, $\mathrm{Re}\lambda \geq 0$ and $\lambda \neq 1$. We begin by transforming Equation \eqref{spectral equation} into a more tractable form. More precisely, we introduce a new independent variable 
					$$
						x=\frac{8\rho^2}{5+3\rho^2}
					$$
				and new dependent variable $y(x;\la)$ defined by the equation
					$$
						f(\rho; \lambda)=:(8-3x)^{\frac{\lambda+3}{2}}y(x; \lambda)
					$$
				which transforms Equation \eqref{spectral equation} into one of Heun type, namely
					\begin{align}
					\begin{split}
						y''(x; \lambda)&+\frac{3(\lambda+5)x^2-(8\lambda+43)x+28}{x(1-x)(8-3x)}y'(x; \lambda)  
						\\
						&+\frac{(\lambda-1)\big(3(\lambda+9)x-5\lambda-51\big)}{4x(1-x)(8-3x)}y(x; \lambda)=0 \label{new spectral eqn}
					\end{split}
					\end{align}
				for $x\in(0,1)$. The solution $f(\cdot;1)$ transforms into $y(\cdot;1)=1$ up to a multiplicative constant. Observe that this transformation preserves smoothness on $[0,1]$, i.e. $f(\cdot; \lambda) \in C^\infty[0,1]$ if and only if  $y(\cdot;\lambda) \in C^\infty[0,1]$. Now, Frobenius theory implies that any smooth solution of the last equation is also analytic on $[0,1]$. We aim to show that $y(\cdot;\lambda)$ must fail to be analytic at $x=1$ unless $\lambda= 1$.
				
The Frobenius indices at $x=0$ are $\{0,-\frac{5}{2}\}$. Without loss of generality, a smooth solution around zero can be written as
					\begin{equation}
						y(x;\lambda)=\sum_{n=0}^\infty a_n(\lambda)x^n,\;\;\;a_0(\lambda)=1 \label{series}
					\end{equation}
				near $x=0$. Now the finite regular singular points of the above Heun equation are $\{0,1,\tfrac{8}{3}\}$. Thus, $y(\cdot;\lambda)$ fails to be analytic at $x=1$ precisely when the radius of convergence of the series \eqref{series} is equal to one, which is what we prove in the following. 
				
By inserting \eqref{series} into Equation \eqref{new spectral eqn}, we find that a recurrence relation for the coefficients $a_n(\lambda)$ is given by
					\begin{equation}
						a_{n+2}(\lambda)=A_n(\lambda)a_{n+1}(\lambda)+B_n(\lambda)a_n(\lambda) \label{recurrence relation}
					\end{equation}
				where
					$$
						A_n(\lambda)=\frac{44n^2+8n(4\lambda+27)+\lambda(5\lambda+78)+121}{16(n+2)(2n+9)}
					$$
					$$
						B_n(\lambda)=-\frac{3(\lambda+2n-1)(\lambda+2n+9)}{16(n+2)(2n+9)}
					$$
				and $a_{-1}(\lambda)=0$, and $a_0(\lambda)=1$. We define
					$$
						r_n(\lambda):=\frac{a_{n+1}(\lambda)}{a_n(\lambda)}.
					$$
				Since $\lim_{n\to\infty}A_n(\lambda)=\frac{11}{8}$, $\lim_{n\to\infty}B_n(\lambda)=-\frac{3}{8}$, the so-called characteristic equation of Equation \eqref{recurrence relation} is 
					$$
						t^2-\tfrac{11}{8}t+\tfrac{3}{8}=0.
					$$
				Solutions of this equation are given by $t_1=\dfrac{3}{8}$ and $t_2=1$. By Poincar\'e's theorem on difference equations (see Theorem 8.9 on p. 343 of \cite{E05} or Appendix A of \cite{GS21}) we conclude that either $a_n(\lambda)=0$ eventually in $n$, 
					\begin{equation}
						\lim_{n\to\infty}r_n(\lambda)=1 \label{potential limit 1}
					\end{equation}
				or
					\begin{equation}
						\lim_{n\to\infty}r_n(\lambda)=\tfrac{3}{8} \label{potential limit 2}.
					\end{equation}
				In fact, $a_n(\lambda)$ cannot go to zero eventually in $n$ since backwards substitution would imply $a_0(\lambda)=0$ which is a contradiction. More precisely, suppose there exists $N\in\N$ such that $a_n(\la)=0$ for all $n\geq N$. Since the zeros of $B_n(\la)$ are negative, we can divide by $B_n(\la)$ to obtain that $a_{N-1}(\la)=0$. Iterating this procedure yields the contradiction. So, we show that Equation \eqref{potential limit 2} cannot hold.
				
				By plugging Equation \eqref{recurrence relation} into the definition of $r_n(\lambda)$, we derive a recurrence relation for $r_n(\lambda)$ given by
					$$
						r_{n+1}(\lambda)=A_n(\lambda)+\frac{B_n(\lambda)}{r_n(\lambda)}
					$$
				with initial condition
					$$
						r_0(\lambda)=\frac{(\lambda-1)(51+5\lambda)}{112}.
					$$
				We define an approximation to $r_n(\lambda)$ given by
					$$
						\tilde{r}_n(\lambda):=\frac{5\lambda^2}{16(n+1)(2n+7)}+\frac{(16n+23)\lambda}{8(n+1)(2n+7)}+\frac{n+3}{n+1}.
					$$
				The quadratic and linear in $\lambda$ terms are obtained by studying the large $|\lambda|$ behavior of $A_n(\lambda)$ while the constant term is put in by hand in order to mimic the small $|\lambda|$ behavior of the first few iterates of $r_n(\lambda)$. Observe that $\lim_{n\to\infty}\tilde r_n(\lambda)=1$. The approximation $\tilde{r}_n(\lambda)$ is intended to behave like $r_n(\lambda)$ for sufficiently large $n$. To show that this is indeed true, we define the quantity 
					$$
						\delta_n(\lambda):=\frac{r_n(\lambda)}{\tilde{r}_n(\lambda)}-1
					$$
				and derive a recurrence relation for it given by 
					$$
						\delta_{n+1}(\lambda)=\varepsilon_n(\lambda)-C_n(\lambda)\frac{\delta_n(\lambda)}{1+\delta_n(\lambda)}
					$$
				where
					$$
						\varepsilon_n(\lambda):=\frac{A_n(\lambda)\tilde{r}_n(\lambda)+B_n(\lambda)}{\tilde{r}_n(\lambda)\tilde{r}_{n+1}(\lambda)}-1
					$$
				and
					$$
						C_n(\lambda):=\frac{B_n(\lambda)}{\tilde{r}_n(\lambda)\tilde{r}_{n+1}(\lambda)},
					$$
				by again plugging the recurrence relation for $r_n(\lambda)$ into the definition of $\delta_n(\lambda)$. Direct calculation shows that we have the following explicit expressions for $\varepsilon_n(\lambda)$ and $C_n(\lambda)$ given by
\[
C_n(\lambda)=\frac{P_1(n;\lambda)}{P_2(n;\lambda)}, \quad 
						\varepsilon_n(\lambda)=\frac{P_3(n;\lambda)}{P_2(n;\lambda)} \]
				where
					$$
						P_1(n;\lambda):=-48(1+n)(7+2n)(-1+2n+\lambda)(9+2n+\lambda),
					$$
					\begin{align*}
						P_2(n;\lambda):=&\big(336+32n^2+16n(13+2\lambda)+\lambda(46+5\lambda)\big) 
   						\\
   						&\times\big(576+32n^2+16n(17+2\lambda)+\lambda(78+5\lambda)\big),
					\end{align*}
				and
					\begin{align*}
						P_3(n;\lambda):=&-32(7+2n)(669+322n+40n^2)+2\big(11809+4n(2742+467n)\big)\lambda
						\\
						&+\big(2611+4n(178+9n)\big)\lambda^2.
					\end{align*}
				By direct calculation, we see that for $\lambda\in\overline{\mathbb{H}}\setminus\{1\}$, $\varepsilon_n(\lambda)\to0$ and $C_n(\lambda)\to-\tfrac{3}{8}$ as $n\to\infty$.
				
				Now, for $\lambda\in\overline{\mathbb{H}}\setminus\{1\}$, we have the following estimates
				
					\begin{align}
					\begin{split}
						&|\delta_{20}(\lambda)|\leq\tfrac{1}{4},\quad |C_n(\lambda)|\leq\tfrac{3}{8}, \quad |\varepsilon_n(\lambda)|\leq\tfrac{1}{8} \label{bounds}
					\end{split}
					\end{align}
for $n\geq 20$. We discuss the proof of the second estimate since the other two are obtained by the same argument. First, we establish the desired estimate on the imaginary line. Then we can extend the estimate to $\overline{\mathbb H}$ via the Phragm\'en-Lindel\"of principle so long as $C_n(\lambda)$ is analytic and polynomially bounded there. So, observe that for $t\in\mathbb R$, the inequality $|C_n(it)|\leq\frac{3}{8}$ is equivalent to the inequality $64|P_1(n,it)|^2-9|P_2(n,it)|^2\leq0$. For $t\in\mathbb R$ and $n\geq20$, a direct calculation shows that the coefficients of $64|P_1(n,it)|^2-9|P_2(n,it)|^2$ are manifestly negative which establishes the desired estimate on the imaginary line. Now, we aim to extend the estimate to all of $\overline{\mathbb{H}}$. As $C_n(\lambda)$ is a rational function of polynomials in $\mathbb Z[n,\lambda]$, it is polynomially bounded. Furthermore, a direct calculation of the zeros of $P_2(n,\lambda)$ shows that they are contained in $\mathbb C\setminus\overline{\mathbb H}$ implying the analyticity of $C_n(\lambda)$ in $\overline{\mathbb{H}}$. Thus, the Phragm\'en-Lindel\"of principle extends the estimate to all of $\overline{\mathbb{H}}$. 
					
				With these bounds in hand, we can prove the same bound for $\delta_n$, $n>20$ by induction. Suppose the estimate holds for some $k>20$. Then
					$$
						|\delta_{k+1}(\lambda)|\leq\tfrac{1}{8}+\tfrac{3}{8}\tfrac{\frac{1}{4}}{1-\frac{1}{4}}=\tfrac{1}{4}
					$$
				by the triangle inequality, Equation \eqref{bounds}, and the induction hypothesis. This bound on $\delta_n(\lambda)$ is now sufficient to exclude Equation \eqref{potential limit 2}. To see this, suppose to the contrary that Equation \eqref{potential limit 2} holds. Then
					$$
						\tfrac{1}{4}\geq|\delta_n(\lambda)|=\Big|1-\tfrac{r_n(\lambda)}{\tilde{r}_n(\lambda)}\Big|\to_{n\to\infty}\tfrac{5}{8}
					$$
				which is clearly a contradiction. Thus, Equation \eqref{potential limit 1} must hold and so $y(\cdot;\la)$ fails to be analytic at $x=1$.
		\end{proof}

		\begin{proposition} \label{point spectrum of evolution}
		There is an $\omega_0 >0$ such that 
				$$
					\sigma(\mathbf{L}_\mathbf V)\subseteq\{\lambda\in\mathbb{C}:\Re\lambda \leq -\omega_0 \}\cup\{1\}.
				$$
Furthermore, the eigenvalue $1$ has a one-dimensional eigenspace, i.e., $\ker(1-\mathbf{L}_\mathbf V)=\langle\mathbf f_1^*\rangle$ where
				$$
					\mathbf f_1^*(\xi)=
						\begin{pmatrix}
							f_{1,1}^*(\xi)
							\\
							f_{1,2}^*(\xi)
						\end{pmatrix}:=
						\begin{pmatrix}
							f(|\xi|;1)
							\\
							|\xi| f'(|\xi|;1)+2f(|\xi|;1)
						\end{pmatrix}.
				$$
		\end{proposition}
		\begin{proof}
			Direct calculation shows that $\mathbf f_1^*\in\mathcal D(\tilde{\mathbf L}_0)$ and that $(1-\mathbf L_\mathbf V)\mathbf f_1^*=\mathbf 0$. Lemma \ref{Le:Spectum_Struct}, Equation \eqref{spectral equation} and Proposition \ref{mode stability}  imply the inclusion.
			
			To see that the eigenspace is one-dimensional and spanned by $\mathbf f_1^*$, suppose $\mathbf u=(u_1,u_2)\in\ker(1-\mathbf L_\mathbf V)$. Direct calculation shows that the equation $(1-\mathbf{L}_\mathbf V)\mathbf{u}=0$ implies that the radial representative of $u_1$ solves the ODE
					\begin{equation}
							-(1-\rho^2)\hat u_1''(\rho)- \Big(\tfrac{6}{\rho}-8\rho\Big)\hat u_1'(\rho)+\Big(10-\tilde V(\rho)\Big)\hat u_1(\rho)=0 \label{spectral 1 eqn}
					\end{equation}
				for $\rho\in(0,1)$ with its second component given by $\hat u_2(\rho)=\rho\hat u_1'(\rho)+2\hat u_1(\rho)$. From our previous calculations, we know that $f(\cdot;1)$ from Proposition \ref{mode stability}  solves Equation \eqref{spectral 1 eqn}. A second linearly independent solution is given explicitly by
\begin{align}\label{Def:g1}
						g_1(\rho):=\frac{375+2125\rho^2+10425\rho^4+243\rho^6+6144\rho^5\log(1-\rho)-6144\rho^5 \log(1+\rho)}{3\rho^5\big(5+3\rho^2\big)^2}.
\end{align}
				Thus, the general solution of Equation \eqref{spectral 1 eqn} is given by
					$$
						u_1(\rho)=c_1f(\r;1)+c_2g_1(\rho)
					$$
				for constants $c_1,c_2,\in\mathbb C$. However, the general solution fails to be in the Sobolev space $H^5_\text{rad}(\mathbb{B}^7)$ unless $c_2=0$ due to the logarithmic behavior at $\r=1$. Thus, $\ker(1-\mathbf{L}_\mathbf V)\subseteq\langle\mathbf f_1^*\rangle$.
\end{proof}

\subsubsection{Semigroup bounds}
		
Since $\lambda = 1$ is an isolated eigenvalue, we can define the corresponding Riesz projection.

\begin{definition}
				Let $\gamma:[0,2\pi]\to\mathbb C$ be defined by $\gamma(t)=1+\frac{1}{2}e^{it}$. Then we set
					$$
						\mathbf P_\mathbf V:=\frac{1}{2\pi i}\int_\gamma\mathbf{R}_{\mathbf L_\mathbf V}(\lambda)d\lambda.
					$$
			\end{definition}

			\begin{proposition} \label{projection}
The projection $\mathbf P_\mathbf V$ commutes with $\big(\mathbf{S}_\mathbf V(\tau)\big)_{\tau\geq0}$ for all $\tau \geq 0$. Furthermore,  $\range\mathbf P_\mathbf V=\langle\mathbf f_1^*\rangle$ and for any $\mathbf u \in\mathcal H$ and all $\tau\geq0$
\begin{equation}
\mathbf{S}_\mathbf V(\tau)\mathbf P_\mathbf V\mathbf{u}=e^\tau\mathbf P_\mathbf V\mathbf{u}. \label{growth on projection}
\end{equation} 
Finally, there exists $\omega>0$ and $C\geq1$ such that
\begin{equation}
									\|\mathbf S_\mathbf V(\t)(1-\mathbf P_\mathbf V)\mathbf u\|_\mathcal H\leq Ce^{-\omega \t}\|(1-\mathbf P_\mathbf V)\mathbf u\|_\mathcal H \label{decay on stable subspace}
\end{equation}
for any $\mathbf u \in\mathcal H$ and all $\tau\geq0$.
			\end{proposition}
			
			\begin{proof}
				By definition, $\mathbf P_\mathbf V$ commutes with $\mathbf{L}_\mathbf V$ and thus commutes with the semigroup $\mathbf{S}_\mathbf V(\tau)$, see \cite{K95}. Next, we show that $\langle\mathbf f_1^*\rangle=\range\mathbf P_\mathbf V$. In fact, it suffices to show $\range\mathbf P_\mathbf V\subseteq\langle\mathbf f_1^*\rangle$ since the reverse inclusion follows from abstract theory. To see this, first observe that $\mathbf P_\mathbf V$ decomposes the Hilbert space as $\mathcal{H}=\range\mathbf P_\mathbf V\oplus\ker\mathbf P_\mathbf V$. The operator $\mathbf{L}_\mathbf V$ is decomposed into the parts $\mathbf{L}_1$ and $\mathbf{L}_2$ on the range and kernel of $\mathbf P_\mathbf V$ respectively. The spectra of these operators are given by
					$$
						\sigma(\mathbf{L}_2)=\sigma(\mathbf{L_\mathbf V})\setminus\{1\},\;\;\;\sigma(\mathbf{L}_1)=\{1\}.
					$$

By Proposition \ref{Le:Spectum_Struct}, the algebraic multiplicity of $1$ is finite, i.e.,  $\rank\mathbf P_\mathbf V:=\dim\range\mathbf P_\mathbf V<\infty$. Hence, the operator $1-\mathbf{L}_1$ acts on the finite-dimensional Hilbert space $\range\mathbf P_\mathbf V$ and, since $\sigma(\mathbf{L}_1)=\{1\}$, $0$ is the only spectral point of $1-\mathbf{L}_1$. Thus, $1-\mathbf{L}_1$ is nilpotent, i.e., there exists $k\in\mathbb{N}$ such that 
					$$
						(1-\mathbf{L}_1)^k\mathbf{u}=\mathbf 0
					$$
				for all $\mathbf{u}\in\range\mathbf P_\mathbf V$ where $k$ is minimal. If $k=1$, then $\langle\mathbf f_1^*\rangle=\range\mathbf P_\mathbf V$ by Proposition \ref{point spectrum of evolution}. Suppose $k\geq2$. Then there exists $\mathbf{u}\in\range\mathbf P_\mathbf V$ such that $(1-\mathbf{L}_1)\mathbf{u}\neq\mathbf 0$ but $(1-\mathbf{L}_1)^2\mathbf{u}=\mathbf 0$. Thus $(1-\mathbf{L})\mathbf{u}=\alpha\mathbf f_1^*$ for some $\alpha\in\mathbb{C}\setminus\{0\}$. Without loss of generality, we set $\alpha=-1$. Observe that the radial representative of the first component of $\mathbf{u}$ solves the ODE
					\begin{equation}
						-(1-\rho^2)\hat u_1''(\rho)-\big(\tfrac{6}{\rho}-8\rho\big)\hat u_1'(\rho)+\Big(10-\hat V(\rho)\Big)\hat u_1(\rho)=G(\rho) \label{inhom spectral 1 eqn}
					\end{equation}
				for $\rho\in(0,1)$ where
					$$
						G(\rho)=\frac{3\rho^2-35}{\big(5+3\rho^2\big)^3}.
					$$
				Recall that we have a fundamental system $\{f(\cdot;1),g_1\}$ of the homogeneous equation, see Proposition \ref{mode stability} and Equation \eqref{Def:g1} for the definitions. Their Wronskian is given explicitly by
					$$
						W\big(f(\cdot;1),g_1\big)(\rho)=\rho^{-6}(1-\rho^2)^{-1}=:W(\r).
					$$
				By variation of parameters, the general solution of \eqref{inhom spectral 1 eqn} can be expressed as
					\begin{align*}
						u_1(\rho)=&c_1f(\r;1)+c_2g_1(\rho)
						\\
						&-g_1(\rho)\int_0^\rho\frac{f(s;1)}{W(s)}\frac{G(s)}{1-s^2}ds+f(\r;1)\int_0^\rho\frac{g_1(s)}{W(s)}\frac{G(s)}{1-s^2}ds
					\end{align*}
				for some $c_1,c_2\in\mathbb{C}$ and all $\rho\in(0,1)$. Explicitly, we find
					$$
						\int_0^\rho\frac{f(s;1)}{W(s)}\frac{G(s)}{1-s^2}ds=-\frac{\rho^7}{(5+3\rho^2)^4}.
					$$
				Consequently, demanding $u_1\in H^5_\text{rad}(\mathbb{B}^7)$ implies we must have $c_2=0$. Thus, we are left with 
					$$
						\hat u_1(\rho)=c_1f(\r;1)+\frac{\rho^7g_1(\rho)}{(5+3\rho^2)^4}+f(\r;1)\int_0^\rho\frac{g_1(s)}{W(s)}\frac{G(s)}{1-s^2}ds.
					$$
				Inspection of the explicit expressions reveals that the remaining integral indeed converges as $\r\to1^-$. Thus, $u_1$ fails to be in $H^5_\text{rad}(\mathbb{B}^7)$ due to the logarithmic behavior of $g_1$ near $\rho=1$ in the second term. We conclude that there is no such solution in $H^5_\text{rad}(\mathbb{B}^7)$ and, as a consequence, we must have $k=1$.
				
Now, observe that Equation \eqref{growth on projection} follows from the facts that $\lambda=1$ is an eigenvalue of $\mathbf{L}_\mathbf V$ with eigenfunction $\mathbf f_1^*$ and $\range\mathbf P_\mathbf V=\langle\mathbf f_1^*\rangle$.
Finally, the growth bound \eqref{decay on stable subspace} is a consequence of the resolvent bounds in Proposition \ref{Le:Spectum_Struct} and the Gearhart-Pr\"uss-Greiner Theorem (see Theorem 1.11 on p. 302 of \cite{EN00}). 
\end{proof}

\subsection{Main result on the linear time evolution}\label{Main result on the linear time evolution}
		We are now in a position to prove our main result on the evolution described by $\mb L$. First, we have as an immediate consequence of Proposition \ref{point spectrum of evolution} that
\[
				\sigma(\mathbf L)\subseteq\{\la\in\mathbb C:\Re\la \leq-\omega_0\}\cup\{1\}
\]
		with $1$ being an eigenvalue. Furthermore, $\ker(1-\mathbf L)=\langle\mathbf\Gamma^{-1}\mathbf f_1^*\rangle$. We write $\mb g_1^*:=\mathbf\Gamma^{-1}\mathbf f_1^*$ and denote the corresponding Riesz projection by 
			$$
				\mathbf P :=\frac{1}{2\pi i}\int_\gamma\mathbf{R}_{\mathbf L}(\lambda)d\lambda. 
			$$

The following statement is a direct consequence of Proposition \ref{projection}.
		
			\begin{theorem}\label{linear stability}
				The projection $\mathbf P$ commutes with the semigroup $\big(\mathbf{S}(\tau)\big)_{\t\geq0}$ and satisfies $\range\mathbf P=\langle\mathbf g_1^*\rangle$. Furthermore, 
								$$
									\mathbf{S}(\tau)\mathbf{P}\mathbf{u}=e^\tau\mathbf{P}\mathbf{u},
								$$
for any $\mathbf u\in\mathcal H$ and all $\tau\geq0$. Finally, for $\omega>0$ as in Proposition \ref{projection}, there exists $C'\geq1$ so that
								$$
									\|\mathbf S(\t)(1-\mathbf P)\mathbf u\|_\mathcal H\leq C'e^{-\omega\t}\|(1-\mathbf P)\mathbf u\|_\mathcal H
								$$
							for any $\mathbf u\in\mathcal H$ and all $\tau\geq0$.
			\end{theorem}
			\begin{proof}
				According to Proposition \ref{compactification of L}, we have that $\mathbf\Gamma\mathbf P\mathbf\Gamma^{-1}=\mathbf P_\mathbf V$. That $\range\mathbf P=\langle\mathbf g_1^*\rangle$ follows from the fact that the map $\mathbf\Gamma^{-1}:\range\mathbf P_\mathbf V\to\range\mathbf P$ is a bijection. By Corollary \ref{Similar Semigroup} and Proposition \ref{projection} we obtain
					\begin{align*}
						\|\mathbf S(\t)(1-\mathbf P)\mathbf f\|_\mathcal H=&\|\mathbf\Gamma^{-1}\mathbf S_\mathbf V(\t)\mathbf\Gamma(1-\mathbf\Gamma^{-1}\mathbf P_\mathbf V\mathbf\Gamma)\mathbf f\|_\mathcal H
						\\
						&=\|\mathbf\Gamma^{-1}\mathbf S_\mathbf V(\t)(1-\mathbf P_\mathbf V)\mathbf\Gamma\mathbf f\|_\mathcal H
						\\
						&\leq C\|\mathbf\Gamma^{-1}\|e^{-\omega\t}\|(1-\mathbf P_\mathbf V)\mathbf\Gamma\mathbf f\|_\mathcal H
						\\
						&=C\|\mathbf\Gamma^{-1}\|e^{-\omega\t}\|\mathbf\Gamma\mathbf\Gamma^{-1}(1-\mathbf P_\mathbf V)\mathbf\Gamma\mathbf f\|_\mathcal H
						\\
						&\leq C\|\mathbf\Gamma^{-1}\|\|\mathbf\Gamma\|e^{-\omega\t}\|(1-\mathbf P)\mathbf f\|_\mathcal H.
					\end{align*}
				Setting $C':=C\|\mathbf\Gamma^{-1}\|\|\mathbf\Gamma\|\geq1$ establishes the claim.
			\end{proof}

\section{The nonlinear time evolution}\label{The nonlinear time evolution}
	This section is devoted to solving the nonlinear problem \eqref{modified expanded first order eqn}. For the remainder of the arguments, we restrict our attention to the real-valued subspace of $\mc H$. We begin by showing that, within our functional analytic framework, the nonlinearity defines a locally Lipschitz mapping on sufficiently small balls in $\mathcal H$. Then, by a contraction mapping argument, we construct solutions of the nonlinear problem. First, we perform some preliminary calculations and decompositions.
	
	\subsection{Nonlinear estimates}\label{Nonlinear estimates}
		For $\mathbf u=(u_1,u_2)\in C_\text{rad}^\infty(\overline{\mathbb B^7})\times C_\text{rad}^\infty(\overline{\mathbb B^7})$, the nonlinearity $\mathbf N$ is given by the expression
			$$
				\mathbf{N}(\mathbf{u})(\xi):=
					\begin{pmatrix}
						0
						\\
						N\big(|\xi|u_1(\xi),\xi^j\partial_j u_1(\xi),|\xi|u_2(\xi),|\xi|\big)
					\end{pmatrix}
			$$
		for $\xi\in\overline{\mathbb B^7}$ with $N$ defined as in Equation \eqref{scalar nonlinear remainder}. Given $\d>0$ and $k\in\mathbb N$, we define
			$$
				\mathcal B_\delta^k:=\{\mathbf u\in\mathcal H^k:\|\mathbf u\|_{\mathcal H^k}\leq\delta\}.
			$$
		If $k=5$, then we will simply write $\mc B_\d:=\mc B_\d^5$. The goal of this section is to prove the following proposition.
			\begin{proposition}\label{locally lipschitz nonlinearity estimate}
				Let $k\in\mathbb N$ with $k\geq5$. There exists $\d_0>0$ such that for any $\d\in(0,\d_0]$, the map $\mathbf N:\mathcal B_\d^k\to\mathcal H^k$ is defined and satisfies the following local Lipschitz bound
					$$
						\|\mathbf N(\mathbf u)-\mathbf N(\mathbf v)\|_{\mathcal H^k}\lesssim_k\big(\|\mathbf u\|_{\mathcal H^k}+\|\mathbf v\|_{\mathcal H^k}\big)\|\mathbf u-\mathbf v\|_{\mathcal H^k}.
					$$
			\end{proposition}
		We will prove this by first decomposing the nonlinearity into three pieces and proving the bound on each piece separately. 
		
		\subsubsection{Decomposition of the nonlinearity}
			First, recall the expression 
				\begin{align*}
					F(x,y,z,\r)=&-\r^{-1}\cot(x)\big(z^2-y^2\big)
					\\
					&-2\r^{-2}\big(1-x\cot(x)\big)y
					\\
					&-\r^{-3}\Big(\tfrac{3}{2}\sin(2x)-2x-x^2\cot(x)\Big)
				\end{align*}
			for real numbers $x,y,z,\r$ to be specified later. We decompose this into three terms
				\begin{equation}
					F_1(x):=-\Big(\frac{3}{2}\sin(2x)-2x-x^2\cot(x)\Big), \label{F1}
				\end{equation}
				$$
					F_2(x,y):=-2\big(1-x\cot(x)\big)y,
				$$
			and
				$$
					F_3(x,y,z):=-\cot(x)\big(z^2-y^2\big)
				$$
			so that
				$$
					F(x,y,z,\r)=\r^{-3}F_1(x)+\r^{-2}F_2(x,y)+\r^{-1}F_3(x,y,z).
				$$
			Recall that $N$ is obtained by expanding $F(x,y,z,\r)$ around
				$$
					(x,y,z)=\big(\r U_1(\r)+\r\zeta_1,\r U_1'(\r)+\zeta_2,\r U_2(\r)+\r\zeta_3\big)
				$$
			for real numbers $\zeta_1,\zeta_2,\zeta_3$ to be specified later. To that end, we define
				$$
					\hat N_1(\zeta_1,\r):=\r^{-3}\Big(F_1\big(\r U_1(\r)+\r\zeta_1\big)-F_1\big(\r U_1(\r)\big)-F_1'\big(\r U_1(\r)\big)\r\zeta_1\Big), \label{N1}
				$$
				\begin{align*}
					\hat N_2(\zeta_1,\zeta_2,\r):=\r^{-2}\Big(&F_2\big(\r U_1(\r)+\r\zeta_1,\r U_1'(\r)+\zeta_2\big)-F_2\big(\r U_1(\r),\r U_1'(\r)\big)
					\\
					&-\partial_1F_2\big(\r U_1(\r),\r U_1'(\r)\big)\r\zeta_1-\partial_2F_2\big(\r U_1(\r),\r U_1'(\r)\big)\zeta_2\Big),
				\end{align*}
			and
				\begin{align*}
					\hat N_3(\zeta_1&,\zeta_2,\zeta_3,\r)
					\\
					:=&\r^{-1}\Big(F_3\big(\r U_1(\r)+\r\zeta_1,\r U_1'(\r)+\zeta_2,\r U_2(\r)+\r\zeta_3\big)
					\\
					&-F_3\big(\r U_1(\r),\r U_1'(\r),\r U_2(\r)\big)-\partial_1F_3\big(\r U_1(\r),\r U_1'(\r),\r U_2(\r)\big)\r\zeta_1
					\\
					&-\partial_2F_3\big(\r U_1(\r),\r U_1'(\r),\r U_2(\r)\big)\zeta_2-\partial_3F_3\big(\r U_1(\r ),\r U_1'(\r),\r U_2(\r)\big)\zeta_3\Big) 
				\end{align*}
			so that
				$$
					N(\r\zeta_1,\zeta_2,\r\zeta_3,\r)=\hat N_1(\zeta_1,\r)+\hat N_2(\zeta_1,\zeta_2,\r)+\hat N_3(\zeta_1,\zeta_2,\zeta_3,\r).
				$$
			For $\mathbf u=(u_1,u_2)\in C_\text{rad}^\infty(\overline{\mathbb B^7})\times C_\text{rad}^\infty(\overline{\mathbb B^7})$, we define
				$$
					\mathbf{N}_1(\mathbf u)(\xi):=
						\begin{pmatrix}
							0
							\\
							N_1\big(u_1(\xi),\xi\big)
						\end{pmatrix},
				$$
				$$
					\mathbf{N}_2(\mathbf u)(\xi):=
						\begin{pmatrix}
							0
							\\
							N_2\big(u_1(\xi),\xi^j\partial_j u_1(\xi),\xi\big)
						\end{pmatrix},
				$$
			and
				$$
					\mathbf{N}_3(\mathbf u)(\xi):=
						\begin{pmatrix}
							0
							\\
							N_3\big(u_1(\xi),\xi^j\partial_j u_1(\xi),u_2(\xi),\xi\big)
						\end{pmatrix}
				$$
			where
				$$
					N_1\big(u_1(\xi),\xi\big)=\hat N_1\big(u_1(\xi),|\xi|),
				$$
				$$
					N_2\big(u_1(\xi),\xi^j\partial_j u_1(\xi),\xi\big)=\hat N_2\big(u_1(\xi),\xi^j\partial_j u_1(\xi),|\xi|),
				$$
			and
				$$
					N_3\big(u_1(\xi),\xi^j\partial_j u_1(\xi),u_2(\xi),\xi\big)=\hat N_3\big(u_1(\xi),\xi^j\partial_j u_1(\xi),u_2(\xi),|\xi|)
				$$
			so that
				$$
					\mathbf{N}=\mathbf{N}_1+\mathbf{N}_2+\mathbf{N}_3.
				$$
			We proceed by proving local Lipschitz bounds on $\mathbf N_1,\mathbf N_2,$ and $\mathbf N_3$ separately. 

			\subsection{Estimates on $\mathbf N_1$}\label{Estimates on N_1}
				We begin with the nonlinear expression $N_1$. By Taylor's theorem with integral remainder, we can write
					$$
						\hat N_1(\zeta_1,\r)=\frac{1}{2}\r^{-1}F_1''\big(\r U_1(\r)\big)\zeta_1^2+\frac{1}{2}\zeta_1^3\int_0^1F_1^{(3)}\big(\r U_1(\r)+t\r\zeta_1\big)(1-t)^2dt.
					$$
				In this form, we begin by proving that the nonlinearity is defined for smooth, radial functions on balls of radius $R\in[1,2]$ satisfying a certain smallness condition.
					
					\begin{lemma} \label{N1 well defined}
						For each $R\in[1,2]$, there exists $\d_0>0$ sufficiently small so that if $\d\in(0,\d_0]$ and $u\in C^\infty_\text{rad}(\overline{\mathbb B^7_R})$ with $\|u\|_{H^5(\mathbb B^7_R)}\leq\d$, then
							$$
								N_1\big(u(\cdot),\cdot\big)\in C^\infty_\text{rad}(\overline{\mathbb B^7_R}).
							$$
					\end{lemma}
					\begin{proof}
						Observe that the expression $F_1(x)$ given by Equation \eqref{F1} is defined for $0<|x|<\pi$. A direct calculation verifies that $F_1$ has a removable discontinuity at $x=0$ and that $\lim_{x\to0}F_1(x)=0$. Thus, we extend the domain of $F_1$ to include $x=0$ by setting $F_1(0)=0$. In particular, we have that $F_1\in C^\infty(-\pi,\pi)$.
		
						A direct calculation shows that $\max_{\r\in[0,2]}\r U_1(\r)<\pi$. Upon imposing the condition 
							\begin{equation}
								|\zeta_1|\leq\frac{1}{2}\big(\pi-\max_{\r\in[0,2]}\r U_1(\r)\big)=:A, \label{max of profile}
							\end{equation}
						we ensure the function $\hat N_1:[-A,A]\times[0,R]\to\mathbb R$ given by $(\zeta_1,\r)\mapsto \hat N_1(\zeta_1,\r)$ is defined. 
						
						To ensure $N_1\big(u(\xi),\xi\big)$ yields finite values for $\xi\in\overline{\mathbb B_R^7}$, it suffices to have
							$$
								\|u\|_{L^\infty(\mathbb B_R^7)}\leq A.
							$$ 
						The Sobolev embedding $H^5(\mathbb B_R^7)\hookrightarrow L^\infty(\mathbb B_R^7)$ allows us to conclude that $\|u\|_{L^\infty(\mathbb B_R^7)}\lesssim_R\d_0$. Thus, it is possible to take $\d_0$ sufficiently small to obtain finite values as desired. 
				
						Now, a direct calculation verifies that $(\cdot)^{-1}F_3''\big((\cdot)U_1\big)\in C_e^\infty[0,R]$. Finally, for any $f\in C_o^\infty[0,1]$, it follows that $F_3^{(3)}(f)\in C_e^\infty[0,R]$ from which the claim follows.
					\end{proof}
			
				Having defined $N_1\big(u(\cdot),\cdot\big)$ for $u\in C^\infty_\text{rad}(\overline{\mathbb B_R^7})$, we proceed to prove local Lipschitz bounds on $\mathbf N_1$ from small balls in $\mathcal H^k$ for any $k\geq5$.
					
					\begin{proposition} \label{N1 lipschitz bound}
						Let $k\in\mathbb N$ with $k\geq5$. There exists $\d_0>0$ such that for any $\d\in(0,\d_0]$, the map $\mathbf N_1:\mathcal B_\d^k\to\mathcal H^k$ is defined and satisfies the following local Lipschitz bound
							$$
								\|\mathbf N_1(\mathbf u)-\mathbf N_1(\mathbf v)\|_{\mathcal H^k}\lesssim_k\big(\|\mathbf u\|_{\mathcal H^k}+\|\mathbf v\|_{\mathcal H^k}\big)\|\mathbf u-\mathbf v\|_{\mathcal H^k}.
							$$
					\end{proposition}
					\begin{proof}
						In what follows, we note that all of the pointwise expressions are defined due to the Sobolev embedding $H^k(\mathbb B^7)\hookrightarrow C^{k-4}(\mathbb B^7)$ for $k\geq4$.
						
						We prove this by an application of Lemma 2.13 of \cite{BDS19}. To that end, we fix two smooth cutoff functions $\chi_1:\R\to\R$ and $\chi_2:\mathbb R^7\to\R$ with the properties that 
							\begin{enumerate}
								\item $\chi_1(\zeta_1)=1$ for $|\zeta_1|\leq\frac{A}{2}$, $\chi_1(\zeta_1)=0$ for $|\zeta_1|\geq\frac{2A}{3}$, $\chi_1$ decreases smoothly in the transition region, and
								\item $\chi_2(\xi)=1$ for $|\xi|\leq\frac{3}{2}$, $\chi_2(\xi)=0$ for $|\xi|\geq\frac{5}{3}$, $\chi_2$ decreases smoothly and radially in the transition region.
							\end{enumerate}
						Now, consider the auxiliary quantity $\mathcal N_{1}:\R\times\mathbb R^7\to\R$ defined by 
							$$
								\mathcal N_1(\zeta_1,\xi):=
									\begin{cases}
										\chi_1(\zeta_1)\chi_2(\xi)N_1(\zeta_1,\xi),& (\zeta_1,\xi)\in[-A,A]\times\mathbb B_2^7
										\\
										0,&(\zeta_1,\xi)\in\mathbb R\times\mathbb R^7\setminus\big([-A,A]\times\mathbb B_2^7\big)
									\end{cases}.
							$$
						A direct calculation verifies that $\mathcal N_1\in C^\infty(\mathbb R\times\mathbb R^7)$ and that $\mathcal N_1(0,\xi)=\partial_1\mathcal N_1(0,\xi)=0$ for all $\xi\in\mathbb R^7$. Thus, Lemma 2.13 of \cite{BDS19} implies
							$$
								\big\|\mathcal N_1\big(u_1(\cdot),\cdot\big)-\mathcal N_1\big(v_1(\cdot),\cdot\big)\big\|_{H^{k-1}(\mathbb B^7)}\lesssim_k\big(\|u_1\|_{H^{k-1}(\mathbb B^7)}+\|v_1\|_{H^{k-1}(\mathbb B^7)}\big)\|u_1-v_1\|_{H^{k-1}(\mathbb B^7)}.
							$$
						By the Sobolev embedding $H^5(\mathbb B^7)\hookrightarrow L^\infty(\mathbb B^7)$, we can take $\d_0$ from Lemma \ref{N1 well defined} with $R=1$ smaller if necessary to ensure
							$$
								\|u\|_{L^\infty(\mathbb B^7)}\leq\frac{A}{2}
							$$
						for all $u\in\mathcal B_\d^k$. The claim then follows after noting that $\mathcal N_1\big(u_1(\xi),\xi\big)=N_1\big(u_1(\xi),\xi\big)$ for all $\xi\in\mathbb B^7$ and $\mathbf u\in\mathcal B_\d^k$.
			\end{proof}

		\subsection{Estimates on $\mathbf N_2$} \label{Estimates on N_2}
				We continue with the nonlinear expression $N_2$. For ease of notation, we set 
					$$
						\mu_2(x):=-2\big(1-x\cot(x)\big)
					$$
				so that
					$$
						F_2(x,y)=\mu_2(x)y.
					$$
				By Taylor's theorem with integral remainder, we write
					\begin{align}
					\begin{split}
						\hat N_2(\zeta_1,\zeta_2,\r)=&\r^{-1}\mu_2'\big(\r U_1(\r)\big)\zeta_1\zeta_2+\frac{1}{2}\r U_1'(\r)\mu_2''\big(\r U_1(\r)\big)\zeta_1^2
						\\
						&+\zeta_1^2\zeta_2\int_0^1\mu_2''\big(\r U_1(\r)+t\r \zeta_1\big)(1-t)dt
						\\
						&+\frac{1}{2}\r U_1'(\r)\zeta_1^3\int_0^1\r\mu_2^{(3)}\big(\r U_1(\r)+t\r \zeta_1\big)(1-t)^2dt. \label{integral form of hN2}
					\end{split}
					\end{align}
				In this form, we can follow the calculations in Section \ref{Estimates on N_1} and prove that this nonlinear term is also defined for smooth, radial functions with minor modifications. For a smooth function $u$, we write $\Lambda u(\xi):=\xi^j\partial_ju(\xi)$.
					
					\begin{lemma} \label{N2 well defined}
						For each $R\in[1,2]$, there exists $\d_0>0$ sufficiently small so that if $\d\in(0,\d_0]$ and $u\in C^\infty_\text{rad}(\overline{\mathbb B^7_R})$ with $\|u\|_{H^5(\mathbb B^7_R)}\leq\d$, then
							$$
								N_2\big(u(\cdot),\Lambda u(\cdot),\cdot\big)\in C^\infty_\text{rad}(\overline{\mathbb B^7_R}).
							$$
					\end{lemma}
					\begin{proof}
						As in the beginning of proof of Lemma \ref{N1 well defined}, we extend the domain of $\mu_2$ to include $0$ by setting $\mu_2(0)=0$ so that we have $\mu_2\in C^\infty(-\pi,\pi)$. For each $R\in[1,2]$, we can choose $\d_0>0$ as in Lemma \ref{N1 well defined} to ensure
							$$
								\|u\|_{L^\infty(\mathbb B_R^7)}\leq A.
							$$ 
						Thus, according to Equation \eqref{integral form of hN2}, $N_2\big(u(\xi),\xi^j\partial_j u(\xi),\xi\big)$ is defined for all $\xi\in\overline{\mathbb B_R^7}$. Direct calculations verify that 
							$$
								(\cdot)U_1'(\cdot),(\cdot)^{-1}\mu_2'\big((\cdot)U_1(\cdot)\big),\mu_2''\big((\cdot)U_1(\cdot)\big)\in C_e^\infty[0,R].
							$$ 
						Lastly, for any $f\in C_o^\infty[0,1]$, it follows that $\mu_2''(f),(\cdot)\mu_2^{(3)}(f)\in C_e^\infty[0,R]$ from which the claim follows.
					\end{proof}
			
				Having defined $N_2\big(u(\cdot),\Lambda u(\cdot),\cdot\big)$ for $u\in C^\infty_\text{rad}(\overline{\mathbb B_R^7})$, we prove local Lipschitz bounds on $\mathbf N_2$ from small balls in $\mathcal H^k$ for any $k\geq5$ as follows.
					
					\begin{proposition} \label{N2 lipschitz bound}
						Let $k\in\mathbb N$ with $k\geq5$. There exists $\d_0>0$ such that for any $\d\in(0,\d_0]$, the map $\mathbf N_2:\mathcal B_\d^k\to\mathcal H^k$ is defined and satisfies the following local Lipschitz bound
							$$
								\|\mathbf N_2(\mathbf u)-\mathbf N_2(\mathbf v)\|_{\mathcal H^k}\lesssim_k\big(\|\mathbf u\|_{\mathcal H^k}+\|\mathbf v\|_{\mathcal H^k}\big)\|\mathbf u-\mathbf v\|_{\mathcal H^k}.
							$$
					\end{proposition}
					\begin{proof}
						Take $\d_0$ as in Proposition \ref{N1 lipschitz bound}. Using the cutoff functions from the proof of Proposition \ref{N1 lipschitz bound}, consider the auxiliary quantity $\mathcal N_{2}:\R\times\R\times\mathbb R^7\to\R$ defined by 
							$$
								\mathcal N_2(\zeta_1,\zeta_2,\xi):=
									\begin{cases}
										\chi_1(\zeta_1)\chi_2(\xi)N_2(\zeta_1,\zeta_2,\xi),& (\zeta_1,\zeta_2,\xi)\in[-A,A]\times\R\times\mathbb B_2^7
										\\
										0,&(\zeta_1,\zeta_2,\xi)\in\R\times\mathbb R\times\mathbb R^7\setminus\big([-A,A]\times\R\times\mathbb B_2^7\big)
									\end{cases}.
							$$
						A direct calculation verifies that $\mathcal N_2\in C^\infty(\mathbb R\times\mathbb R\times\mathbb R^7)$ and that $\mathcal N_2(0,0,\xi)=\partial_1\mathcal N_2(0,0,\xi)=\partial_2\mathcal N_2(0,0,\xi)=0$ for all $\xi\in\mathbb R^7$. Repeating the argument from the proof of Proposition \ref{N1 lipschitz bound} on any term in Equation \eqref{integral form of hN2} not involving $\zeta_2$ yields the desired bound. Thus, it remains to establish the desired bound for the remaining terms, i.e.,
							$$
								\chi_1(\zeta_1)\chi_2(\xi)\bigg(|\xi|^{-1}\mu_2'\big(|\xi| U_1(|\xi|)\big)\zeta_1\zeta_2+\zeta_1^2\zeta_2\int_0^1\mu_2''\big(|\xi| U_1(|\xi|)+t|\xi| \zeta_1\big)(1-t)dt\bigg).
							$$
						For the first term, we write
							\begin{align*}
								\chi_2(\xi)&|\xi|^{-1}\mu_2'\big(|\xi| U_1(|\xi|)\big)\Big(\chi_1\big(u_1(\xi)\big)u_1(\xi)\xi^j\partial_j u_1(\xi)-\chi_1\big(v_1(\xi)\big)v_1(\xi)\xi^j\partial_j v_1(\xi)\Big)
								\\
								=&\chi_2(\xi)|\xi|^{-1}\mu_2'\big(|\xi| U_1(|\xi|)\big)\chi_1\big(u_1(\xi)\big)u_1(\xi)\Big(\xi^j\partial_j u_1(\xi)-\xi^j\partial_j v_1(\xi)\Big)
								\\
								&+\chi_2(\xi)|\xi|^{-1}\mu_2'\big(|\xi| U_1(|\xi|)\big)\Big(\chi_1\big(u_1(\xi)\big)u_1(\xi)-\chi_1\big(v_1(\xi)\big)v_1(\xi)\Big)\xi^j\partial_j v_1(\xi).
							\end{align*}
						By our choice of $\delta_0$ and the algebra property of $H^{k-1}(\mathbb B^7)$ for $k\geq4$, taking an $H^{k-1}(\mathbb B^7)$-norm yields
							\begin{align*}
								\Big\|\chi_2&(\cdot)|\cdot|^{-1}\mu_2'\big(|\cdot| U_1(|\cdot|)\big)\Big(\chi_1\big(u_1(\cdot)\big)u_1(\cdot)\Lambda u_1(\cdot)-\chi_1\big(v_1(\cdot)\big)v_1(\cdot)\Lambda v_1(\cdot)\Big)\Big\|_{H^{k-1}(\mathbb B^7)}
								\\
								\lesssim&\|u_1\|_{H^{k-1}(\mathbb B^7)}\|u_1-v_1\|_{H^{k}(\mathbb B^7)}+\|v_1\|_{H^{k}(\mathbb B^7)}\|u_1-v_1\|_{H^{k-1}(\mathbb B^7)}
							\end{align*}
						for all $\mathbf u,\mathbf v\in\mc B_\d^k$. For the final term, we note that this term is of the form $\zeta_2N(\zeta_1,\xi)$ with $N$ satisfying the desired local Lipschitz bound using the same argument as in the proof of Proposition \ref{N1 lipschitz bound}. Thus, upon writing
							\begin{align*}
								\xi^j\partial_j u_1&(\xi)N\big(u_1(\xi),\xi\big)-\xi^j\partial_j v_1(\xi)N\big(v_1(\xi),\xi\big)
								\\
								=&\big(\xi^j\partial_j u_1(\xi)-\xi^j\partial_j v_1(\xi)\big)N\big(u_1(\xi),\xi\big)-\xi^j\partial_j v_1(\xi)\Big(N\big(u_1(\xi),\xi\big)-N\big(v_1(\xi),\xi\big)\Big)
							\end{align*}
						and taking an $H^{k-1}(\mathbb B^7)$-norm, we obtain
							\begin{align*}
								\|\xi^j\partial_j u_1&N\big(\xi,u_1\big)-\xi^j\partial_j v_1N\big(\xi,v_1\big)\|_{H^{k-1}(\mathbb B^7)}
								\\
								&\lesssim\|u_1\|_{H^{k-1}(\mathbb B^7)}\|u_1-v_1\|_{H^{k}(\mathbb B^7)}+\|v_1\|_{H^{k}(\mathbb B^7)}\|u_1-v_1\|_{H^{k-1}(\mathbb B^7)}.
							\end{align*}
						The claim then follows as a consequence.
			\end{proof}
								
		\subsection{Estimates on $\mathbf N_3$}\label{Estimates on N_3}
				We end our nonlinear estimates with the nonlinear expression $N_3$. Again for notational convenience, we write $\mu_3(x)=-\cot(x)$
				so that
					\[
						F_3(x,y,z)=\mu_3(x)\big(z^2-y^2\big).
					\]
				By Taylor's theorem with integral remainder, we write
					\begin{align*}
						\hat N_1(\zeta_1,\zeta_2,\zeta_3,\r)=&\r\mu_3\big(\r U_1(\r)+\r\zeta_1\big)\zeta_3^2-\r^{-1}\mu_3\big(\r U_1(\r)+\r\zeta_1\big)\zeta_2^2
						\\
						&+2U_2(\r)\r^2\mu_3'\big(\r U_1(\r)\big)\zeta_1\zeta_3-2\r^{-1}U_1'(\r)\r^2\mu_3'\big(\r U_1(\r)\big)\zeta_1\zeta_2
						\\
						&+\Big(U_2(\r)^2-U_1'(\r)^2\Big)\frac{1}{2}\r^3\mu_3''\big(\r U_1(\r)\big)\zeta_1^2
						\\
						&+2U_2(\r)\zeta_1^2\zeta_3\int_0^1\r^3\mu_3''\big(\r U_1(\r)+t\r\zeta_1\big)dt
						\\
						&-2\r^{-1}U_1'(\r)\zeta_1^2\zeta_2\int_0^1\r^3\mu_3''\big(\r U_1(\r)+t\r\zeta_1\big)dt
						\\
						&+\Big(U_2(\r)^2-U_1'(\r)^2\Big)\frac{1}{2}\zeta_1^3\int_0^1\r^4\mu_3^{(3)}\big(\r U_1(\r)+t\r\zeta_1\big)dt
					\end{align*}
				Again, we follow the calculations in Sections \ref{Estimates on N_1} and \ref{Estimates on N_2} to prove that this nonlinear term is also well-defined for smooth, radial functions with minor modifications.
					
					\begin{lemma} \label{N23well defined}
						For each $R\in[1,2]$, there exists $\d_0>0$ sufficiently small so that if $\d\in(0,\d_0]$ and $u_1,u_2\in C^\infty_\text{rad}(\overline{\mathbb B^7_R})$ with $\|(u_1,u_2)\|_{H^5(\mathbb B^7_R)\times H^4(\mathbb B^7_R)}\leq\d$, then
							$$
								N_3\big(u_1(\cdot),\Lambda u_1(\cdot),u_2(\cdot)\cdot\big)\in C^\infty_\text{rad}(\overline{\mathbb B^7_R}).
							$$
					\end{lemma}
					\begin{proof}
						For this, we use crucially that we only consider real-valued radial functions and that $U_1(\r)>0$ for $\r\in[0,2]$ and attains a positive minimum in $[0,2]$. First, note that for smooth, radial functions $u$, it always holds that
							$$
								\big(\xi^j\partial_j u(\xi)\big)^2=|\xi|^2|\nabla u(\xi)|^2.
							$$ 
						Thus, this nonlinear term can be equivalently expressed as
							\begin{align*}
								N_1\big(&u_1(\xi),\xi^j\partial_j u_1(\xi),u_2(\xi),\xi\big)
								\\
								=&|\xi|\mu_3\big(|\xi| U_1(|\xi|)+|\xi|u_1(\xi)\big)u_2(\xi)^2
								\\
								&-|\xi|\mu_3\big(|\xi| U_1(|\xi|)+|\xi|u_1(\xi)\big)|\nabla u_1(\xi)|^2
								\\
								&+2U_2(|\xi|)|\xi|^2\mu_3'\big(|\xi| U_1(|\xi|)\big)u_1(\xi)u_2(\xi)
								\\
								&-2|\xi|^{-1}U_1'(|\xi|)|\xi|^2\mu_3'\big(|\xi| U_1(|\xi|)\big)u_1(\xi)\xi^j\partial_j u_1(\xi)
								\\
								&+\Big(U_2(|\xi|)^2-U_1'(|\xi|)^2\Big)\frac{1}{2}|\xi|^3\mu_3''\big(|\xi| U_1(|\xi|)\big)u_1(\xi)^2
								\\
								&+2U_2(|\xi|)u_1(\xi)^2u_2(\xi)\int_0^1|\xi|^3\mu_3''\big(|\xi| U_1(|\xi|)+|\xi|u_1(\xi)\big)dt
								\\
								&-2|\xi|^{-1}U_1'(|\xi|)u_1(\xi)^2\xi^j\partial_j u_1(\xi)\int_0^1|\xi|^3\mu_3''\big(|\xi| U_1(|\xi|)+t|\xi|u_1(\xi)\big)dt
								\\
								&+\Big(U_2(|\xi|)^2-U_1'(|\xi|)^2\Big)\frac{1}{2}u_1(\xi)^3\int_0^1|\xi|^4\mu_3^{(3)}\big(|\xi| U_1(|\xi|)+t|\xi|u_1(\xi)\big)dt.
							\end{align*}
						We claim that $(\cdot)^\ell\mu_3^{(\ell-1)}\big((\cdot)U_1(\cdot)+(\cdot)\zeta_1\big)\in C_e^\infty[0,R]$ for all $\ell\in\mathbb N$. We demonstrate this for $\ell=1$ as higher values of $\ell$ follow analogously. For $\r\in[0,R]$, we write
							\begin{align*}
								\r\mu_3\big(\r U_1(\r)+\r\zeta_1\big)=\frac{1}{U_1(\r)+\zeta_1}\big(\r U_1(\r)+\r\zeta_1\big)\mu_3\big(\r U_1(\r)+\r\zeta_1\big)
							\end{align*}
						since $U_1(\r)+\zeta_1\neq0$. A direct calculation shows that if $f\in C_o^\infty[0,R]$, then $f\mu_3(f)\in C_e^\infty[0,R]$. By our choice of $\delta_0$, we ensure that $\big(U_1(|\cdot|)+u_1\big)^{-1}\in C_\text{rad}^\infty(\overline{\mathbb B_R^7})$. Direct calculations furthermore verify that
							$$
								U_2,(\cdot)^{-1}U_1',(U_1')^2,(\cdot)^2\mu_3'\big((\cdot)U_1\big),(\cdot)^3\mu_3''\big((\cdot)U_1\big)\in C_e^\infty[0,R]
							$$
						from which the claim follows.
					\end{proof}
			
				Having defined $N_3\big(u_1(\cdot),\Lambda u_1(\cdot),u_2(\cdot)\cdot\big)$ for $(u_1,u_2)\in C^\infty_\text{rad}(\overline{\mathbb B_R^7})\times C^\infty_\text{rad}(\overline{\mathbb B_R^7})$, we proceed to prove local Lipschitz bounds on $\mathbf N_3$ from small balls in $\mathcal H^k$ for any $k\geq5$ as follows.
					
					\begin{proposition} \label{N3 lipschitz bound}
						Let $k\in\mathbb N$ with $k\geq5$. There exists $\d_0>0$ such that for any $\d\in(0,\d_0]$, the map $\mathbf N_3:\mathcal B_\d^k\to\mathcal H^k$ is defined and satisfies the following local Lipschitz bound
							$$
								\|\mathbf N_3(\mathbf u)-\mathbf N_3(\mathbf v)\|_{\mathcal H^k}\lesssim_k\big(\|\mathbf u\|_{\mathcal H^k}+\|\mathbf v\|_{\mathcal H^k}\big)\|\mathbf u-\mathbf v\|_{\mathcal H^k}.
							$$
					\end{proposition}
					\begin{proof}
						Again, take $\d_0$ as in Proposition \ref{N1 lipschitz bound}. Furthermore, using the cutoff functions from the proof of Proposition \ref{N1 lipschitz bound}, we consider the auxiliary quantity $\mathcal N_3:\R\times\R\times\R\times\mathbb R^7\to\R$ defined by 
							\begin{align*}
								\mathcal N_3(\zeta_1&,\zeta_2,\zeta_3,\xi)
								\\	:=&\begin{cases}
										\chi_1(\zeta_1)\chi_2(\xi)N_3(\zeta_1,\zeta_2,\zeta_3,\xi),& (\zeta_1,\zeta_2,\zeta_3,\xi)\in[-A,A]\times\R\times\R\times\mathbb B_2^7
										\\
										0,&(\zeta_1,\zeta_2,\zeta_3,\xi)\in\R\times\R\times\mathbb R\times\mathbb R^7\setminus\big([-A,A]\times\R\times\R\times\mathbb B_2^7\big)
									\end{cases}.
							\end{align*}
						A direct calculation verifies that $\mathcal N_3\in C^\infty(\mathbb R\times\mathbb R\times\R\times\mathbb R^7)$ and that $\mathcal N_3(0,0,0,\xi)=\partial_1\mathcal N_3(0,0,0,\xi)=\partial_2\mathcal N_3(0,0,0,\xi)=\partial_3\mathcal N_3(0,0,0,\xi)=0$ for all $\xi\in\mathbb R^7$. The claim then follows with minor modifications using the arguments from the proof of Proposition \ref{N2 lipschitz bound}.
			\end{proof}
		
		Finally, we prove the main result on the nonlinearity.
		
			\begin{proof}[Proof of Proposition \ref{locally lipschitz nonlinearity estimate}]
				The claim follows by the triangle inequality and Propositions \ref{N1 lipschitz bound}, \ref{N2 lipschitz bound}, and \ref{N3 lipschitz bound}.
			\end{proof}			
	
	\subsection{The abstract Cauchy problem} \label{The abstract cauchy problem}
		We turn our attention to studying the abstract initial value problem
			\begin{equation}
				\begin{cases}
					\partial_\t\Phi(\t)=\mathbf L\Phi(\t)+\mathbf N\big(\Phi(\t)\big),&\t>0
					\\
					\Phi(0)=\mb u
				\end{cases} \label{abstract ivp}
			\end{equation}
		for $\mb u\in\mathcal B_\delta$ for any $\delta\leq\delta_0$ as in Proposition \ref{locally lipschitz nonlinearity estimate}. Using the semigroup, we reformulate this as an integral equation via Duhamel's formula
			$$
				\Phi(\tau)=\mathbf S(\tau)\mb u+\int_0^\tau\mathbf S(\tau-s)\mathbf N\big(\Phi(s)\big)ds
			$$
		on the Banach space
			$$
				\mathcal X:=\{\Phi\in C([0,\infty),\mathcal H):\|\Phi\|_\mathcal X:=\sup_{\t>0}e^{\omega\t}\|\Phi(\t)\|_\mathcal H<\infty\}
			$$
		for $\omega>0$ as in Theorem \ref{linear stability}. However, due to $1\in\sigma_p(\mathbf L)$, it is not possible to prove the existence of a solution in the space $\mathcal X$ for small data $\mb u$. To remedy this, we first consider a modified problem following the Lyapunov-Perron method from dynamical systems theory. Given $\Phi\in\mc X$ and $\mathbf u\in\mc B_\d$, we introduce a correction term
			$$
				\mathbf C(\Phi,\mathbf u):=\mathbf P\Big(\mathbf u+\int_0^\infty e^{-s}\mathbf N\big(\Phi(s)\big)ds\Big)
			$$
		and consider the modified equation
			\begin{equation}
				\Phi(\t)=\mathbf S(\tau)\big(\mathbf u-\mathbf C(\Phi,\mathbf u)\big)+\int_0^\tau\mathbf S(\tau-s)\mathbf N\big(\Phi(s)\big)ds. \label{modified Duhamel}
			\end{equation}
		We will first show the existence of a unique solution of Equation \eqref{modified Duhamel} within the space $\mathcal X$ and, afterward, show that this correction term can be suppressed by taking $\mathbf u$ as in Equation \eqref{Initia_Data_Trans} and allowing the blowup time to vary. 
		
		\begin{proposition} \label{mod wp}
			For all sufficiently large $c>0$ and sufficiently small $\delta>0$ and any $\mathbf u\in\mathcal H$ satisfying $\|\mathbf u\|_{\mathcal H}\leq\frac{\delta}{c}$, there exists a unique solution $\Phi_{\mathbf u}\in C([0,\infty),\mathcal H)$ of Equation \eqref{modified Duhamel} that satisfies $\|\Phi_{\mathbf u}(\t)\|_{\mathcal H}\leq\delta e^{-\omega\t}$ for all $\t\geq0$. Furthermore, the solution map $\mathbf u\mapsto\Phi_{\mathbf u}$ is Lipschitz as a map from $\mc B_{\delta/c}$ to $\mathcal X$.
		\end{proposition}
		\begin{proof}
			Introduce the closed ball
				$$
					\mathcal X_\delta:=\{\Phi\in C([0,\infty),\mathcal H):\|\Phi\|_\mathcal X\leq\delta\}
				$$
			and formally define the map
				$$
					\mathbf K_\mathbf u(\Phi)(\t):=\mathbf S(\tau)\big(\mathbf u-\mathbf C(\Phi,\mathbf u)\big)+\int_0^\tau\mathbf S(\tau-s)\mathbf N\big(\Phi(s)\big)ds.
				$$
			By taking $\d_0$ small enough, we can ensure that for any $\Phi=(\varphi_1,\varphi_2)\in\mathcal X_\d$ we have
				$$
					\sup_{\t\geq0}\|\varphi_1(\t)\|_{L^\infty(\mathbb B^7)}\leq\frac{A}{2}
				$$ 
			by the Sobolev embedding $H^5(\mathbb B^7)\hookrightarrow L^\infty(\mathbb B^7)$ where $A$ is the number defined in \eqref{max of profile}. We aim to show that $\mathbf K_{\mathbf u}:\mathcal X_\delta\to\mathcal X_\delta$ is a well-defined contraction map. 
				
			First, observe that by Theorem \ref{linear stability} and Proposition \ref{projection}, we have
				$$
					\mathbf P\mathbf K_{\mathbf u}(\Phi)(\t)=-\int_\t^\infty e^{\t-s}\mathbf P\mathbf N\big(\Phi(s)\big)ds.
				$$
			From Proposition \ref{locally lipschitz nonlinearity estimate} and the fact that $\mathbf N(\mathbf0)=\mathbf0$, we have the estimate
				\begin{align*}
					\|\mathbf P\mathbf K_{\mathbf u}(\Phi)(\t)\|_{\mathcal H}&\lesssim e^{\t}\int_\t^\infty e^{-s}\|\Phi(s)\|_{\mathcal H}^2ds
					\\
					&\lesssim e^{\t}\|\Phi\|_{\mathcal X}^2\int_\t^\infty e^{-s-2\omega s}ds' \lesssim \delta^2e^{-2\omega \t}.
				\end{align*}
			By Proposition \ref{projection}, we have $(1-\mathbf P)\mathbf C(\Phi,\mathbf u)=\mathbf 0$. This implies
				$$
					(1-\mathbf P)\mathbf K_{\mathbf u}(\Phi)(\t)=\mathbf S(\t)(1-\mathbf P)\mathbf u+\int_{0}^\t\mathbf S(\t-s)(1-\mathbf P)\mathbf N\big(\Phi(s)\big)ds.
				$$
			By Theorem \ref{linear stability}, we obtain
				\begin{align*}
					\|(1-\mathbf P)\mathbf K_{\mathbf u}(\Phi)(\t)\|_{\mathcal H}&\lesssim e^{-\omega \t}\|(1-\mathbf P)\mathbf u\|_{\mathcal H}+\int_{0}^\t e^{-\omega (\t-s)}\|\mathbf N\big(\Phi(s)\big)\|_{\mathcal H}ds
					\\
					&\lesssim\frac{\delta}{c}e^{-\omega \t}+e^{-\omega \t}\int_{0}^se^{\omega s}\|\Phi(s)\|_{\mathcal H}^2ds
					\\
					&\lesssim\frac{\delta}{c}e^{-\omega \t}+\|\Phi\|_{\mathcal X}^2e^{-\omega \t}\int_{0}^\t e^{-\omega s}ds
					\\
					&\lesssim\frac{\delta}{c}e^{-\omega \t}+\delta^2e^{-\omega \t}
				\end{align*}
			for all $\t\geq0$. Thus, for $\delta_0$ sufficiently small and $c$ sufficiently large, we can ensure
				$$
					\|\mathbf K_{\mathbf u}(\Phi)(\t)\|_{\mathcal H}\leq\delta e^{-\omega \t}.
				$$
			Consequently, we see that $\mathbf K_{\mathbf u}:\mathcal X_\delta\to\mathcal X_\delta$. 
				
			We claim that $\mathbf K_{\mathbf u}$ is a contraction map. Given $\Phi,\Psi\in\mathcal X_\delta$, observe that
				$$
					\mathbf P\mathbf K_{\mathbf u}(\Phi)(\t)-\mathbf P\mathbf K_{\mathbf u}(\Psi)(\t)=-\int_\t^\infty e^{\t-s}\mathbf P\Big(\mathbf N\big(\Phi(s)\big)-\mathbf N\big(\Psi(s)\big)\Big)ds.
				$$
			By Proposition \ref{locally lipschitz nonlinearity estimate}, we have that 
				\begin{align*}
					\|\mathbf P\mathbf K_{\mathbf u}(\Phi)(\t)&-\mathbf P\mathbf K_{\mathbf u}(\Psi)(\t)\|_{\mathcal H}
					\\
					&\lesssim e^{\t}\int_\t^\infty e^{-s}\big(\|\Phi(s)\|_{\mathcal H}+\|\Psi(s)\|_{\mathcal H}\big)\|\Phi(s)-\Psi(s)\|_{\mathcal H}ds
					\\
					&\lesssim\delta\|\Phi-\Psi\|_{\mathcal X}e^{\t}\int_\t^\infty e^{-s-2\omega s}ds
					\lesssim\delta e^{-2\omega \t}\|\Phi-\Psi\|_{\mathcal X}.
				\end{align*}
			Furthermore,
				$$
					(1-\mathbf P)\mathbf K_{\mathbf u}(\Phi)(\t)-(1-\mathbf P)\mathbf K_{\mathbf u}(\Psi)(\t)=\int_0^\t\mathbf S(\t-s)(1-\mathbf P)\Big(\mathbf N\big(\Phi(s)\big)-\mathbf N\big(\Psi(s)\big)\Big)ds.
				$$
			By Theorem \ref{linear stability} and Proposition \ref{locally lipschitz nonlinearity estimate}, we obtain
				\begin{align*}
					\|(1-\mathbf P)\mathbf K_{\mathbf u}(\Phi)(\t)&-(1-\mathbf P)\mathbf K_{\mathbf u}(\Psi)(\t)\|_{\mathcal H}
					\\
					&\lesssim\int_0^\t e^{-\omega (\t-s)}\big(\|\Phi(s)\|_{\mathcal H}+\|\Psi(s)\|_{\mathcal H}\big)\|\Phi(s)-\Psi(s)\|_{\mathcal H}ds
					\\
					&\lesssim\delta\|\Phi-\Psi\|_{\mathcal X}e^{-\omega \t}\int_0^\t e^{-\omega s}ds
					\lesssim\delta e^{-\omega \t}\|\Phi-\Psi\|_{\mathcal X}.
				\end{align*}
			Thus,
				$$
					\|\mathbf K_{\mathbf u}(\Phi)-\mathbf K_{\mathbf u}(\Psi)\|_{\mathcal X}\lesssim\delta\|\Phi-\Psi\|_{\mathcal X}
				$$
			and by considering smaller $\delta_0$ if necessary, we see that $\mathbf K_{\mathbf u}$ is a contraction on $\mathcal X_\delta$. The Banach fixed point theorem implies the existence of a unique fixed point $\Phi_\mathbf u\in\mathcal X_\delta$ of $\mathbf K_\mathbf u$. 
			
			We now claim that the solution map $\mathbf u\mapsto\Phi_\mathbf u$ is Lipschitz. Observe that
				\begin{align*}
					\|\Phi_\mathbf u-\Phi_\mathbf v\|_{\mathcal X}&=\|\mathbf K_\mathbf u(\Phi_\mathbf u)-\mathbf K_\mathbf v(\Phi_\mathbf v)\|_{\mathcal X}
					\\
					&\leq\|\mathbf K_\mathbf u(\Phi_\mathbf u)-\mathbf K_\mathbf u(\Phi_\mathbf v)\|_{\mathcal X}+\|\mathbf K_\mathbf u(\Phi_\mathbf v)-\mathbf K_\mathbf v(\Phi_\mathbf v)\|_{\mathcal X}
					\\
					&\lesssim\delta\|\Phi_\mathbf u-\Phi_\mathbf v\|_{\mathcal X}+\|\mathbf K_\mathbf u(\Phi_\mathbf v)-\mathbf K_\mathbf v(\Phi_\mathbf v)\|_{\mathcal X}.
				\end{align*}
			A direct calculation shows
				$$
					\mathbf K_\mathbf u(\Phi_\mathbf v)(\t)-\mathbf K_\mathbf v(\Phi_\mathbf v)(\t)=\mathbf S(\t)(1-\mathbf P)(\mathbf u-\mathbf v).
				$$
			Theorem \ref{linear stability} yields
				$$
					\|\mathbf K_\mathbf u(\Phi_\mathbf v)(\t)-\mathbf K_\mathbf v(\Phi_\mathbf v)(\t)\|_{\mathcal H}\lesssim e^{-\omega \t}\|\mathbf u-\mathbf v\|_{\mathcal H}.
				$$
			Thus, we have
				$$
					\|\Phi_\mathbf u-\Phi_\mathbf v\|_{\mathcal X}\lesssim\delta\|\Phi_\mathbf u-\Phi_\mathbf v\|_{\mathcal X}+\|\mathbf u-\mathbf v\|_{\mathcal H}.
				$$
			Again, considering smaller $\delta_0$ if necessary yields the result. Finally, that $\Phi_\mathbf u$ is the unique solution in $\mc X$ follows by standard arguments on unconditional uniqueness.
		\end{proof}
			
	\subsection{Variation of the blowup time}\label{Reconnecting to the Physical Problem}
		In this section, we show that the correction term in Equation \eqref{modified Duhamel} can be made to vanish by appropriately varying the blowup time $T$. As a first step, we define the \textit{initial data operator}. For functions $\mb v=(v_1,v_2)\in C^1_\text{rad}(\overline{\mathbb B^7_R})\times C_\text{rad}(\overline{\mathbb B^7_R})$, $R>0$, we define the rescaling operator
			$$
				\mathcal R(\mb v,T)(\xi):=\begin{pmatrix}
	   				T v_1(T\xi) 
					\\
	  				T^2 v_2(T\xi) 
				\end{pmatrix}
			$$
		for $\xi\in\overline{\mathbb B^7}$. We write
	$$
		\mb U(\xi):=
			\begin{pmatrix}
				U_1(|\xi|)
				\\
				U_2(|\xi|)
			\end{pmatrix}
	$$
to denote the blowup solution in similarity coordinates. For $T$ in some interval containing $1$ to be specified, we define the initial data operator as
			$$
				\Phi_0(\mb v,T)(\xi)=\mathcal R(\mb v,T)(\xi)+\mathcal R(\mb U,T)(\xi)-\mathcal R(\mb U,1)(\xi).
			$$
		Observe that this is precisely the right-hand side of Equation \eqref{Initia_Data_Trans}. Furthermore, consider the Hilbert space
			$$
				\mc Y:=H^6_\text{rad}(\mathbb B^7_2)\times H^5_\text{rad}(\mathbb B^7_2)
			$$ 
		with the standard norm and denote by $\mc B_\mc Y$ the unit ball in $\mc Y$. We have the following mapping properties of the initial data operator.
		\begin{lemma} \label{id operator}
			The initial data operator $\Phi_0:\mc B_\mc Y\times[\frac{1}{2},\frac{3}{2}]\to\mathcal H$ is Lipschitz continuous, i.e., 
				$$
					\|\Phi_0(\mb v,T_1)-\Phi_0(\mb w,T_2)\|_\mc H\lesssim\|\mb v-\mb w\|_{\mc Y}+|T_1-T_2|
				$$
			for all $\mb v,\mb w\in\mc B_\mc Y$ and $T_1,T_2\in[\frac{1}{2},\frac{3}{2}]$. Furthermore, if $\d\in(0,\d_0]$ for $\d_0>0$ sufficiently small and $\|\mb v\|_{\mc Y}\leq\d$, then for all $T\in[1-\d,1+\d]$,
				$$
					\|\Phi_0(\mb v,T)\|_\mathcal H\lesssim\d.
				$$
		\end{lemma}	
		\begin{proof}
			Observe that the embedding $\mc Y\hookrightarrow C^2(\overline{\mathbb B^7_{2}})\times C^1(\overline{\mathbb B^7_{2}})$ implies that the pointwise definition of the initial data operator makes sense. 
			
			For any $v\in C^1(\overline{\mathbb B_2^7})$, $T_1,T_2\in[\frac{1}{2},\frac{3}{2}]$, and $\xi\in\overline{\mathbb B^7}$, we write
				$$
					v(T_1\xi)-v(T_2\xi)=(T_1-T_2)\int_0^1\xi^j\partial_j v\Big(\big(T_2+s(T_1-T_2)\big)\xi\Big)ds.
				$$
			Consequently, we obtain
				$$
					\|v(T_1\cdot)-v(T_2\cdot)\|_{H^k(\mathbb B^7)}\lesssim\|v\|_{H^{k+1}(\mathbb B_2^7)}|T_1-T_2|
				$$
			for $k\geq4$. For $\mb v,\mb w\in\mc Y$ and $T_1,T_2\in[\frac{1}{2},\frac{3}{2}]$, we then obtain
				\begin{equation}
					\|\mc R(\mb v,T_1)-\mc R(\mb w,T_2)\|_\mc H\lesssim\|\mb v\|_{\mc Y}|T_1-T_2|+\|\mb v-\mb w\|_\mc Y. \label{v LC}
				\end{equation}
			By smoothness of $\mb U$, we similarly have
				\begin{equation}
					\|\mc R(\mb U,T_1)-\mc R(\mb U,T_2)\|_\mc H\lesssim|T_1-T_2|. \label{mbU T LC}
				\end{equation}
			Thus, Lipschitz continuity of the map $\Phi_0:\mc B\times[\frac{1}{2},\frac{3}{2}]\to\mathcal H$ follows. In particular, if we take any $T\in[1-\d,1+\d]$ and set $T_1=T$, $T_2=1$, then \eqref{mbU T LC} shows that
				$$
					\|\mc R(\mb U,T)-\mc R(\mb U,1)\|_\mc H\lesssim\d.
				$$ 
			Furthermore, taking $\|\mb v\|_\mc Y\leq\d$ and $\mb w=0$, \eqref{v LC} shows
				$$
					\|\mc R(\mb v,T)\|_\mc H\lesssim\d.
				$$
			Thus, the second claim follows.
		\end{proof}
		\begin{lemma} \label{correction is zero}
			Let $\d_0>0$ be sufficiently small. For all $\d\in(0,\d_0]$, $c>0$ sufficiently large, and $\mb v\in\mc Y$ with
				$$
					\|\mb v\|_\mc Y\leq\frac{\d}{c^2},
				$$
			there exists a unique $T\in[1-\frac{\d}{c},1+\frac{\d}{c}]$ and a unique $\Phi\in\mathcal X_\d$ which satisfies
				\begin{equation}
					\Phi(\t)=\mathbf S(\t)\Phi_0(\mb v,T)+\int_0^\t\mathbf S(\t-s)\mathbf N\big(\Phi(s)\big)ds \label{duhamel with id operator}
				\end{equation}
			for all $\t>0$. Moreover, $T$ depends Lipschitz continuously on the data, i.e.,
				$$
					|T(\mb v)-T(\mb w)|\lesssim\|\mb v-\mb w\|_\mc Y
				$$
			for all $\mb v,\mb w\in\mc Y$ as above.
		\end{lemma}
		\begin{proof}
			Lemma \ref{id operator} implies $\|\Phi_0(\mb v,T)\|_{\mathcal H}\lesssim\frac{\d}{c^2}$ for all $T\in[1-\frac{\d}{c},1+\frac{\d}{c}]$. By taking $c$ sufficiently large, we can ensure $\|\Phi_0(\mb v,T)\|_{\mathcal H}\leq\frac{\d}{c}$ for all such $T$. Thus, Proposition \ref{mod wp} implies that for each $T\in[1-\frac{\d}{c},1+\frac{\d}{c}]$, there exists $\Phi_T:=\Phi_{\Phi_0(\mb v,T)}\in\mc X_\d$ which is unique in $\mc X$ and solves 
				$$
					\Phi_T(\t)=\mathbf S(\tau)\big(\Phi_0(\mb v,T)-\mathbf C(\Phi_T,\Phi_0(\mb v,T)\big)+\int_0^\tau\mathbf S(\tau-s)\mathbf N\big(\Phi_T(s)\big)ds
				$$
			for all $\t\geq0$. We aim to show that there exists a unique $T=T(\mathbf v)\in[1-\frac{\d}{c},1+\frac{\d}{c}]$ such that $\mathbf C(\Phi_{T},\Phi_0(\mb v,T)\big)=\mathbf 0$. Since $\rg\mathbf P=\langle\mathbf g_1^*\rangle$, this is equivalent to
				\begin{equation}
					\big(\mathbf C(\Phi_T,\Phi_0(\mb v,T))|\mathbf g_1^*\big)_\mathcal H=0. \label{inner prod is zero}
				\end{equation}
			
			By Taylor expansion, we have
				$$
					\mc R(\mb U,T)-\mc R(\mb U,1)=\kappa(T-1)\mb g_1^*+\mb R(T)
				$$
			for some constant $\kappa\in\mathbb R\setminus\{0\}$ with $\mb R(T)$ denoting the second-order remainder term. For $T_1,T_2\in[1-\d,1+\d]$, a direct calculation shows
				$$
					\|\mb R(T_1)-\mb R(T_2)\|_\mc H\lesssim\d|T_1-T_2|.
				$$
			With this, we write the initial data operator as
				$$
					\Phi_0(\mb v,T)=\mc R(\mb v,T)+\gamma(T-1)\mb g_1^*+\mb R(T).
				$$
			Applying the Riesz projection yields
				$$
					\mb P\Phi_0(\mb v,T)=\mb P\mc R(\mb v,T)+\gamma(T-1)\mb g_1^*+\mb P\mb R(T).
				$$
			Now, we write $T=1+\beta$ and define the following quantity
				$$
					\mb\Sigma_\mb v(\beta):=\mb P\mc R(\mb v,T)+\mb P\mb R(T)+\mb P\mb I(\beta)
				$$
			where
				$$
					\mb I(\beta):=\int_0^\infty e^{-s}\mb N(\Phi_{1+\beta}(s))ds.
				$$
			Thus, Equation \eqref{inner prod is zero} is equivalent to
				$$
					\beta=\Sigma_{\mb v}(\beta)=\tilde\kappa(\mb\Sigma_\mb v(\beta)|\mb g_1^*)_\mc H
				$$
			for some $\tilde\kappa\in\mathbb R\setminus\{0\}$. We aim to show that $\Sigma_{\mb v}:[-\frac{\d}{c},\frac{\d}{c}]\to[-\frac{\d}{c},\frac{\d}{c}]$ is a contraction map.
			
			Direct calculation shows that
				$$
					\Sigma_\mb v(\beta)=O\Big(\frac{\d}{c^2}\Big)+O(\d^2).
				$$
			Thus, for $c>0$ sufficiently large and $\d_0>0$ sufficiently small depending on $c$, we obtain $|\Sigma_\mb v|\leq\frac{\d}{c}$. To see that it is a contraction, let $\beta_1,\beta_2\in[-\frac{\d}{c},\frac{\d}{c}]$ and denote by $\Phi\in\mc X_\d$ the solution corresponding to $T_1=1+\beta_1$ and by $\Psi\in\mc X_\d$ the solution corresponding to $T_2=1+\beta_2$. By Proposition \ref{mod wp} and Lemma \ref{id operator}, we have
				$$
					\|\Phi-\Psi\|_\mc X\lesssim\|\Phi_0(\mb v,T_1)-\Phi_0(\mb v,T_2)\|_\mc H\lesssim|\beta_1-\beta_2|.
				$$
			By Proposition \ref{locally lipschitz nonlinearity estimate}, we obtain 
				$$
					\|\mb P\mb I(\beta_1)-\mb P\mb I(\beta_2)\|_\mc H\lesssim\d|\beta_1-\beta_2|.
				$$
			Since $\mb P\in\mc B(\mc H)$, we obtain 
				$$
					|\Sigma_\mb v(\beta_1)-\Sigma_\mb v(\beta_2)|\lesssim\d|\beta_1-\beta_2|.
				$$
			Upon taking $\d_0>0$ smaller if necessary, we have that $\Sigma_\mb v$ is a contraction. Thus, the Banach fixed point theorem implies the existence of a unique $\beta=\beta(\mb v)\in[-\frac{\d}{c},\frac{\d}{c}]$ such that $\mathbf C(\Phi_{T},\Phi_0(\mb v,T)\big)=\mathbf 0$ with $T=1+\beta$.
			
			Now, we show that the $T$ just obtained depends Lipschitz continuously on the data. For $\mb v,\mb w\in\mc Y$ satisfying the smallness assumption, denote by $\beta_\mb v$ and $\beta_\mb w$ the unique parameters obtained as above. We write
				\begin{align*}
					|\beta_\mb v-\beta_\mb w|&=|\Sigma_\mb v(\beta_\mb v)-\Sigma_\mb w(\beta_\mb w)|
					\\
					&\leq|\Sigma_\mb v(\beta_\mb v)-\Sigma_\mb w(\beta_\mb v)|+|\Sigma_\mb w(\beta_\mb v)-\Sigma_\mb w(\beta_\mb w)|.
				\end{align*}
			For the first term, we obtain 
				$$
					|\Sigma_\mb v(\beta_\mb v)-\Sigma_\mb w(\beta_\mb v)|\lesssim\|\mb v-\mb w\|_\mc Y.
				$$
			For the second term, we obtain
				$$
					|\Sigma_\mb w(\beta_\mb v)-\Sigma_\mb w(\beta_\mb w)|\lesssim\d|\beta_\mb v-\beta_\mb w|.
				$$
			So, by taking $\d_0>0$ sufficiently small, we obtain the desired Lipschitz dependence.
		\end{proof}
	
	\subsection{Upgrade to classical solutions}\label{Classical Solutions}
		We now show that if $\mathbf v\in C^\infty_\text{rad}(\overline{\mathbb B^7_{2}})\times C^\infty_\text{rad}(\overline{\mathbb B^7_{2}})$, then the solution obtained in Lemma \ref{correction is zero} is smooth and a classical solution.
			\begin{proposition} \label{classical soln}
				Let $\d_0>0$ and $c>0$ be as in Lemma \ref{correction is zero}, $\d\in(0,\d_0]$, and $\mb v\in C^\infty_\text{rad}(\overline{\mathbb B^7_{2}})\times C^\infty_\text{rad}(\overline{\mathbb B^7_{2}})$ such that 
					 $$
						\|\mb v\|_\mc Y\leq\frac{\d}{c^2}.
					$$
				Then the unique solution $\Phi$ of Equation \eqref{duhamel with id operator} belongs to $C^\infty([0,\infty)\times\mathbb B^7)\times C^\infty([0,\infty)\times\mathbb B^7)$ and solves Equation \eqref{abstract ivp} classically.
			\end{proposition}
			\begin{proof}
				Denote by $T$ the unique parameter obtained in Lemma \ref{correction is zero} and observe that $\Phi_0(\mb v,T)\in\mathcal H^k$ for all $k\in\mathbb N$. According to Proposition \ref{locally lipschitz nonlinearity estimate}, for each $k\in\mathbb N$, $k\geq5$ and any $\d\in(0,\d_0]$, $\mathbf N:\mathcal B^k_\d\to\mathcal H^k$ is locally Lipschitz. Thus, a standard fixed point argument then yields a local solution of Equation \eqref{duhamel with id operator} in $\mathcal H^k$ for each such $k$. By uniqueness, these solutions are precisely the global solution of Equation \eqref{duhamel with id operator} in $\mathcal H$ from Lemma \ref{correction is zero} on their interval of existence. We claim that these solutions are in fact global solutions in $\mathcal H^k$. Denote by $\mc T_k>0$ the lifespan of the solution $\Phi$ in $\mc H^k$, i.e., we have $\Phi\in C([0,\mc T_k],\mc H^k)$. From Equation \eqref{duhamel with id operator} it follows that
					$$
						\|\Phi(\t)\|_{\mathcal H^k}\lesssim_k1+\int_0^\t\|\Phi(s)\|_{\mathcal H^k}ds
					$$
				for all $\t\in[0,\mc T_k]$. Gr\"onwall's inequality then implies $\|\Phi(\t)\|_{\mc H^k}\leq C_1e^{C_2\mc T_k}$ for all $\t\in[0,\mc T_k]$ and for some $C_1,C_2>0$. Thus, by standard continuation criteria (see, e.g. Theorem 4.3.4 on p. 57 of \cite{CH99}), it must hold that $\mc T_k=\infty$. Furthermore, Sobolev embedding yields $\Phi(\t)\in C^\infty(\mathbb B^7)\times C^\infty(\mathbb B^7)$ for all $\t\geq0$.
				
				To prove regularity in $\t$, we first note that  $\Phi_0(\mb v,T)\in\mathcal D(\mathbf L_0)$. Thus, for each fixed $\ell\in\mathbb N$ with $\ell\geq5$, Proposition 4.3.9 on p. 60 of \cite{CH99} implies that the global solution $\Phi$ of Equation \eqref{duhamel with id operator} is a classical solution, i.e., 
					$$
						\Phi\in C([0,\infty),\mathcal D(\mathbf L_0))\cap C^1([0,\infty),\mathcal H^{\ell})
					$$
				and solves
					\begin{equation}
						\partial_\t\Phi(\t)=\mathbf L\Phi(\t)+\mathbf N(\Phi(\t)) \label{classical eqn_}
					\end{equation}
				for $\t\geq0$ in $\mathcal H^{\ell}$. In fact, by the embedding $\mathcal H\hookrightarrow L^\infty(\mathbb B^7)\times L^\infty(\mathbb B^7)$, we have that Equation \eqref{classical eqn_} holds pointwise. Furthermore, since $\Phi(\t)\in\mathcal H^k$ for all $k\geq\ell$, we in fact have that $\mathbf L$ acts classically on $\Phi(\t)$, i.e., $\mathbf L\Phi(\t)=\tilde{\mathbf L}\Phi(\t)$. As a consequence, it follows that
					$$
						\partial_\t\Phi(\t)=\tilde{\mathbf L}\Phi(\t)+\mathbf N(\Phi(\t)). \label{classical eqn}
					$$
				By the embedding $\mathcal H^{\ell}\hookrightarrow L^\infty(\mathbb B^7)\times L^\infty(\mathbb B^7)$, the $\t$-derivative holds pointwise. Finally, by a generalized version of Schwarz' theorem (see e.g. Theorem 9.41 on p. 235 of \cite{R87}), we can exchange $\t$-derivatives and $\xi$-derivatives upon which the claim follows.
			\end{proof}

	\subsection{Proof of the main result}\label{Proof of the main result}
		\begin{proof}[Proof of Theorem \ref{Th:Reform}]
			Choose $\d,c>0$ as in Lemma \ref{correction is zero} and set $\d':=\frac{\d}{c}$. Furthermore, let $(f,g)\in C_\text{rad}^\infty(\mathbb B^7_2)\times C_\text{rad}^\infty(\mathbb B^7_2)$ satisfy
				$$
					\big\|\big(f,g\big)\big\|_{H^6(\mathbb B^7_2)\times H^5(\mathbb B^7_2)}\leq\frac{\d'}{c}.
				$$
			Then $\mathbf v:=(f,g)\in C^\infty_\text{rad}(\overline{\mathbb B^7_{2}})\times C^\infty_\text{rad}(\overline{\mathbb B^7_{2}})$ satisfies the hypotheses of Lemma \ref{correction is zero} and Proposition \ref{classical soln}. Thus, there is a unique $T\in[1-\d',1+\d']$ depending Lipschitz continuously on $(f,g)$ so that Equation \eqref{duhamel with id operator} has the unique classical solution $\Phi=(\varphi_1,\varphi_2)\in C^\infty([0,\infty)\times\mathbb B^7)\times C^\infty([0,\infty)\times\mathbb B^7)$ with $\Phi\in\mathcal X_{\d'}$. Now, set
				\begin{align*}
					u(t,r):=&\frac{1}{T-t}\bigg[\tilde U\Big(\frac{r}{T-t}\Big)+\varphi\Big(t,\frac{r}{T-t}\Big)\bigg].
				\end{align*}
			with
				$$
					\varphi\Big(t,\frac{r}{T-t}\Big):=\varphi_1\bigg(\log\Big(\frac{T}{T-t}\Big),\frac{r}{T-t}\bigg).
				$$
			By Proposition \ref{classical soln} and the fact that similarity coordinates define a diffeomorphism of the backwards light cone into the infinite cylinder, we have that $u\in C_\text{rad}^\infty(\mathfrak C_T)$. Furthermore, according to Proposition \ref{correction is zero} and the calculations carried out in Section \ref{First-order formulation}, $u$ is indeed the unique solution of Equation \eqref{rescaled semilinear sf skyrme eqn} on $\mathfrak C_T$ satisfying the initial conditions
				$$
					u(0,r)=u^1(0,r)+f(r)
				$$
			and
				$$
					\partial_t u(0,r)=\partial_t u^1(0,r)+g(r).
				$$
			The estimate \eqref{convergence} follows from $\Phi\in\mathcal X_\d'$.
%			
%			$\psi=\psi(t,r):$ solves 
%				$$
%					\sin^2(\psi)\bigg(\partial_t^2\psi-\partial_r^2-\frac{2}{r}\partial_r\psi\bigg)+\frac{1}{2}\sin(2\psi)\Big(\big(\partial_t\psi\big)^2-\big(\partial_r\psi\big)^2+\frac{3\sin^2(\psi)}{r^2}\Big)=0
%				$$
%			in $\mc C_T$ and satisfies the initial condition
%				$$
%					\psi(0,r)=U(r)+f(r),
%				$$
%				$$
%					\partial_t\psi(0,r)=rU'(r)+g(r).
%				$$
%			Lastly, the stated convergence follows from $\Phi\in\mathcal X_\d'$.
		\end{proof}

\appendix	
	\section{Derivation of the equation} \label{Derivation of the Equation}
		Here, we carry out the calculations leading to Equation \eqref{skyrme eom}. Consider the $(1+d)$-dimensional Minkowski space $(\R^{1+d},\eta)$, the $d$-sphere $(\mathbb S^d,h)$, and smooth maps $U:\mathbb R^{1+d}\to\mathbb S^d$. On the domain, we use the coordinates $(x^\mu)_{\mu=0}^d$ with $x^0=t$ and the remaining spatial coordinates we leave unspecified for the moment. On the target, we utilize coordinates $(\Omega^a)_{a=0}^{d-1}=(\psi,\Omega)$ where $\psi$ denotes a particular polar angle and $\Omega=(\Omega^1,\Omega^2,\dots,\Omega^{d-1})$ denotes the remaining angles on $\mathbb S^{d-1}$. 
		We express the metrics as
			$$
				\eta=\eta_{\mu\nu}dx^\mu dx^\nu=-dt^2+\eta_{ij}dx^idx^j
			$$
		and
			$$
				h=h_{ab}d\Omega^ad\Omega^b=d\psi^2+\sin^2(\psi)d\Omega^2
			$$
		with $d\Omega^2$ denoting the standard round metric on $\mathbb S^{d-1}$. From this data, we consider the symmetric $(0,2)$-tensor on $\R^{1+d}$ given by the pullback of $h$ via $U$ and denote it by $U^*h$. Composing this quantity with the inverse Minkowski metric, $\eta^{-1}\circ U^*h$, defines a smoothly-varying linear transformation on each tangent space in Minkowski space. Symmetric polynomials of its eigenvalues define smoothly-varying functions on spacetime which are invariant under the symmetry group of $\eta$. To that end, we denote by
			$$
				\sigma_1(U)=\tr_\eta(U^*h)
			$$
		the first symmetric polynomial of the eigenvalues of $\eta^{-1}\circ U^*h$ and by
			$$
				\sigma_2(U)=\tr_\eta(U^*h)^2-\tr_\eta\big((U^*h)^2\big)
			$$
		the second symmetric polynomial of the eigenvalues of $\eta^{-1}\circ U^*h$. In coordinates, these quantities take the form
			$$
				\sigma_1(U)=\eta^{\mu\nu}h(U)_{ab}\partial_\mu U^a\partial_\nu U^b
			$$
		and
			$$
				\sigma_2(U)=\big(\eta^{\mu\nu}h(U)_{ab}\partial_\mu U^a\partial_\nu U^b\big)^2-\eta^{\mu\rho}\eta^{\nu\sigma}h(U)_{ab}h(U)_{cd}\partial_\rho U^a\partial_\sigma U^b\partial_\mu U^c\partial_\nu U^d
			$$
		where $\eta^{\mu\nu}$ denotes the components of $\eta^{-1}$ in the coordinates $(x^\mu)_{\mu=0}^d$. Being Lorentz-invariant quantities depending on the map $U$, linear combinations of these quantities form candidates for Lagrangians of geometric field theories. For $\alpha,\beta\geq0$, consider the action
			\begin{equation}
				\mc S_{Sky}[U]:=\int_{\R^{1+d}}\Big(\frac{\alpha}{2}\sigma_1(U)+\frac{\beta}{4}\sigma_2(U)\Big)d\eta. \label{skyrme model}
			\end{equation}
		Observe that this is precisely the Skyrme model as described in Section \ref{Introduction}. The case $\beta=0$ yields wave maps into the sphere while the case $\alpha=0$ yields the strong field Skyrme model.
		
		We restrict our attention to co-rotational maps. To that end, we put spherical coordinates on the domain, i.e., we set $(x^i)_{i=1}^d=(r,\omega)$ where $\omega=(\omega^1,\dots,\omega^{d-1})$ denotes an angle on $\mathbb S^{d-1}$. In these coordinates, the Minkowski metric takes the form
			$$
				\eta=-dt^2+dr^2+r^2d\omega^2
			$$ 
		with $d\omega^2$ denoting the standard round metric on $\mathbb S^{d-1}\subset\mathbb R^{d}$. Furthermore, we only consider those $U:\mathbb R^{1+d}\to\mathbb S^d$ of the form
			$$
				U(t,r,\omega)=\big(\psi(t,r),\omega\big)
			$$
		for some function $\psi:\mathbb R\times[0,\infty)\to\mathbb R$. The action \eqref{skyrme model} reduces to 
			$$
				\mc S_{Sk}[U]=\int_{-\infty}^\infty\int_0^\infty\mathcal L[\psi](t,r)dtdr
			$$
		with Lagrangian density
			\begin{align*}
				\mathcal L[\psi](t,r):=C_dr^{d-1}\bigg[\Big(\alpha^2&+\frac{(d-1)\beta^2\sin ^2(\psi)}{r^2}\Big)\Big(-\big(\partial_t\psi\big)^2+\big(\partial_r\psi\big)^2\Big)
				\\
				&+\Big(\alpha^2+\frac{(d-2)\beta^2\sin^2(\psi)}{2r^2}\Big)\frac{4 \sin ^2(\psi)}{r^2}\bigg]
			\end{align*}
		where $C_d>0$ is a constant depending on the dimension coming from the angular portion of the action which plays no crucial role. Critical points formally solve the Euler-Lagrange equation
			$$
				\partial_t\frac{\partial\mathcal L[\psi]}{\partial(\partial_t\psi)}+\partial_r\frac{\partial\mathcal L[\psi]}{\partial(\partial_r\psi)}-\frac{\partial\mathcal L[\psi]}{\partial\psi}=0
			$$
		which takes the form
			\begin{align*}
				\Big(\alpha^2&+\frac{\beta^2(d-1)\sin^2(\psi)}{r^2}\Big)\big(\partial_t^2\psi-\partial_r^2\psi\big)-\frac{d-1}{r}\Big(\alpha^2+\frac{\beta^2(d-3)\sin^2(\psi)}{r^2}\Big)\partial_r\psi
				\\
				&+\frac{(d-1)\sin(2\psi)}{2r^2}\bigg(\alpha^2+\beta^2\Big(\big(\partial_t\psi\big)^2-\big(\partial_r\psi\big)^2+\frac{(d-2)\sin^2(\psi)}{r^2}\Big)\bigg)=0.
			\end{align*}
		Setting $\alpha=0$ and $\beta=1$ yields Equation \eqref{sf skyrme eom}.
		
\section{Proof of proposition \ref{free evolution}}\label{The Free Wave Evolution}

We prove a more general result on the spaces $\mc H^k$ for the purpose of Proposition \ref{classical soln} where certain restriction properties of the semigroup is needed. Proposition \ref{free evolution} then follows by setting $k=5$.

\begin{proposition}
Let $k \geq 3$. The operator $\tilde{\mathbf{L}}_0: \mc D(\tilde{\mathbf{L}}_0) \subset \mc H^k \to \mc H^k$ is closable and its closure $\tilde{\mathbf{L}}_{0,k}: \mc D(\tilde{\mathbf{L}}_{0,k}) \subset \mc H^k \to \mc H^k$ generates a semigroup $(\mb S_{0,k}(\tau))_{\tau \geq 0}$ which satisfies 
\[ \|\mb S_{0,k}(\tau) \mb u \|_{\mc H^{k}} \leq M_k  e^{-\frac{1}{2} \tau} \| \mb u \|_{\mc H^{k}}  \]
for all $\tau \geq 0$ and all $\mb u \in \mc H^k$. Moreover, for any $j \in \N$, the semigroup $(\mb S_{0,k+j}(\tau))_{\tau \geq 0}$ is the restriction of $(\mb S_{0,k}(\tau))_{\tau \geq 0}$ to $\mc H^{k+j}_{\text{rad}}$.
\end{proposition}

\begin{proof}
We apply the Lumer-Phillips theorem which necessitates a suitable dissipative bound. For this, we follow the standard procedure and use an equivalent, but better behaved, inner product on $\mc H^k$ instead. Following \cite{GS21} we define for $k \geq 3$ on $C^k_{\text{rad}}(\overline{\B^7}) \times C^{k-1}_{\text{rad}}(\overline{\B^7})$  
\begin{align*}
(\mb u| \mb v)_1 & :=  4 \int_{\B^7} \partial_i \partial_j \partial_k u_1(\xi) \overline{\partial^i \partial^j \partial^k v_1(\xi)} d\xi
+  4 \int_{\B^7}  \partial_i \partial_j  u_2(\xi) \overline{\partial^i \partial^j v_2(\xi) } d\xi 
\\
& +  4 \int_{\mathbb S^6}  \partial_i \partial_j  u_1(\omega)\overline{ \partial^i \partial^j v_1(\omega) }d\sigma(\omega), \\
(\mb u| \mb v)_2 & := \int_{\B^7} \partial_i \Delta u_1(\xi) \overline{\partial^i  \Delta  v_1(\xi)} d\xi
+ \int_{\B^7}  \partial_i \partial_j  u_2(\xi) \overline{\partial^i \partial^j v_2} (\xi) d\xi 
+ \int_{\mathbb S^6}  \partial_i u_2(\omega)\overline{ \partial^i  v_2(\omega) } d\sigma(\omega), \\
(\mb u| \mb v)_3  & := \int_{\mathbb S^6}  \partial_i u_1(\omega) \overline{ \partial^i  v_1(\omega)} d\sigma(\omega) + \int_{\mathbb S^6}   u_1(\omega)  \overline{v_1(\omega)} d\sigma(\omega) 
+ \int_{\mathbb S^6}  u_2(\omega)  \overline{v_2(\omega) }d\sigma(\omega).
\end{align*}
Furthermore, for $4 \leq j \leq k$,  we use the standard $\dot{H}^j(\mathbb B^7) \times\dot{H}^{j-1}(\mathbb B^7)$ inner products and define
\[ (\mathbf u|\mathbf u)_{\mc E^k} := \sum_{j=1}^{k} (\mathbf u|\mathbf u)_{j} \]
and set  $\|\mathbf u\|_{\mc E^k} :=\sqrt{(\mathbf u|\mathbf u)_{\mc E^k}}$. Using Lemma 3.1 of \cite{GS21}, it follows that
 \[ \|\mathbf u\|_{\mc E^k} \simeq \|\mathbf u\|_{\mathcal H^k}\] for all $\mathbf u \in C^k(\overline{\mathbb B^7}) \times  C^{k-1}(\overline{\mathbb B^7})$. Consequently, this holds in particular on $ C^k_\text{rad}(\overline{\mathbb B^7}) \times  C^{k-1}_\text{rad}(\overline{\mathbb B^7})$. By density,  $\|\cdot \|_{\mc E^k}$ defines an equivalent norm on $\mc H^k$.

We write $\tilde{\mathbf L}_0 = \tilde{\mathbf L}_W +\tilde{\mathbf L}_{D}$, where $\tilde{\mathbf L}_W$ is the standard wave evolution in similarity coordinates as defined in \cite{GS21}, Eq. $(1.14)$  and 
	$$
		\tilde{\mathbf L}_{D} \mb u = 
			\begin{pmatrix}
				0 
				\\
				- 2 u_2
			\end{pmatrix}.
	$$ 
By Lemma $3.2$ in \cite{CGS21} (modulo notation) we have 
\[ 		\mathrm{Re} \sum_{j=1}^{3} (\tilde{\mathbf L}_W \mathbf{u}|\mathbf{u})_{j} \leq-\tfrac{1}{2} \sum_{j=1}^{3}  ( \mathbf{u}|\mathbf{u})_{j} . \]
Furthermore, by emulating the computation in the proof of Lemma 3.3 in \cite{CGS21}, Appendix A, one obtains for $4 \leq j \leq k$ the bound
\begin{align*}
	\mathrm{Re}  (\tilde{\mathbf L}_W \mathbf{u}|\mathbf{u})_{j} \leq (\tfrac{5}{2} - k) ( \mathbf{u}|\mathbf{u})_{j} . 
\end{align*}
Obviously, $\Re(\tilde{\mathbf L}_{D} \mathbf{u}|\mathbf{u})_{\mc E^k} \leq 0$, which implies the dissipative estimate
					$$
						\Re(\tilde{\mathbf L}_0 \mathbf{u}|\mathbf{u})_{\mc E^k} \leq-\tfrac{1}{2}\|\mathbf{u}\|_{\mc E^k}^2
					$$
				for all $\mathbf u\in\mathcal D(\tilde{\mathbf L}_0)$.

Next, we prove that set $\range(\frac{3}{2}-\tilde{\mathbf L}_0)$ is dense in $\mathcal H^{k}_{\text{rad}}$. Let $\mathbf f\in C^\infty_\text{rad}(\overline{\mathbb B^7}) \times C^\infty_\text{rad}(\overline{\mathbb B^7})$. We will show that the equation
\[
						\big(\tfrac{3}{2}-\tilde{\mathbf L}_0\big)\mathbf u=\mathbf f
\]
				is solvable with $\mathbf u=(u_1,u_2)\in\mathcal D(\tilde{\mathbf L}_0)$. In terms of radial representatives this, equation is equivalent to the system of ODEs
\[
						\begin{cases}
							\frac{5}{2}\hat u_1(\r)+\r\hat u_1'(\r)-\hat u_2(\r)=\hat f_1(\r)
							\\
							\frac{11}{2}\hat u_2(\r)-\hat u_1''(\r)-\frac{6}{\r}\hat u_1'(\r)+\r\hat u_2'(\r)=\hat f_2(\r)
						\end{cases}
\]
				for $\r\in(0,1)$. Using the first equation to solve for $\hat u_2$, we see that solving this system of ODEs  reduces to
					\begin{equation}
						(1-\r^2)\hat u_1''(\r)+\Big (\tfrac{6}{\r}-9\r\Big)\hat u_1'(\r)-\tfrac{55}{4}\hat u_1(\r)=g(\r) \label{dense range ODE}
					\end{equation}
				for $\r\in(0,1)$ and where
\[
						g(\r):=-\hat f_2(\r)-\r\hat f_1'(\r)-\tfrac{11}{2}\hat f_1(\r).
\]
				Observe that the homogeneous equation has Frobenius indices $\{0,-5\}$ at $\r=0$ and $\{0,-\frac{1}{2}\}$ at $\r=1$. In fact, an explicit fundamental system for the homogenous equation is given by
					$$
						u_{1,1}(\r):=\r^{-5}(1+\r)^{-\frac{1}{2}}\big(12+6\r+\r^2+2\r^3\big)
					$$
				and
					$$
						u_{1,2}(\r):=\r^{-5}(1-\r)^{-\frac{1}{2}}\big(12-6\r+\r^2-2\r^3\big)
					$$
				with Wronskian
					$$
						W\big(u_{1,1},u_{1,2}\big)(\r)=105\r^{-6}\big(1-\r^2\big)^{-\frac{3}{2}}.
					$$
				Observe that while $u_{1,1}$ takes the index $0$ at $\r=1$, both solutions take the index $-5$ at $\r=0$. In order to solve the inhomogeneous equation, we define a third solution
					$$
						u_{1,0}(\r):=u_{1,1}(\r)-u_{1,2}(\r).
					$$
				Direct calculation shows that this solution takes the index $0$ at $\r=0$. A particular solution of Equation \eqref{dense range ODE} is given by
					\begin{align*}
						\hat u_1(\r)=&-u_{1,0}(\r)\int_\r^1\frac{u_{1,1}(s)}{W\big(u_{1,0},u_{1,1}\big)(s)}\frac{g(s)}{1-s^2}ds-u_{1,1}(\r)\int_0^\r\frac{u_{1,0}(s)}{W\big(u_{1,0},u_{1,1}\big)(s)}\frac{g(s)}{1-s^2}ds
						\\
						=&-u_{1,0}(\r)\int_\r^1u_{1,1}(s)\sqrt{1-s}g_1(s)ds-u_{1,1}(\r)\int_0^\r u_{1,0}(s)\sqrt{1-s}g_1(s)ds
					\end{align*}
				where
\[
g_1(s):=\tfrac{1}{105}s^6\sqrt{1+s}g(s).
\]

By direct calculation, we see that $\hat u_1\in C^\infty(0,1)$. We claim that in fact we have $\hat u_1\in C_e^\infty[0,1]$, i.e., $u_1\in C^\infty_\text{rad}(\overline{\mathbb B^7})$. To verify this claim, we first show that $\hat u_1\in C^\infty(0,1]$. Observe that the second integral converges as $\r\to1^-$ and we call its value $\alpha$. Thus, after inserting the definition of $u_{1,0}(\r)$ in terms of the two other solutions, we obtain an equivalent expression for $\hat u_1(\r)$
					$$
						\hat u_1(\r)=\frac{\tilde u_{1,2}(\r)}{\sqrt{1-\r}}\int_\r^1u_{1,1}(s)\sqrt{1-s}g_1(s)ds-\alpha u_{1,1}(\r)-u_{1,1}(\r)\int_\r^1\tilde u_{1,2}(s)g_1(s)ds
					$$
				where
					$$
						\tilde u_{1,2}(\r):=\r^{-5}\big(12-6\r+\r^2-2\r^3\big).
					$$
				Now, the second and third terms are clearly smooth at $\r=1$. For the first term, we make the substitution $s=\r+(1-\r)t$ to obtain the equivalent form
					$$
						\tilde u_{1,2}(\r)(1-\r)\int_0^1u_{1,1}\big(\r+(1-\r)t\big)g_1\big(\r+(1-\r)t\big)\sqrt{1-t}dt
					$$
				for $\r>0$ from which smoothness at $\r=1$ follows.
				
				Now, we show that  $u_1 \in C_\text{rad}^\infty(\overline{\mathbb B^7})$ for $u_1(\xi) = \hat u_1(|\xi|)$  We first note that our analysis so far shows that $u_1  \in C^\infty(\overline{\mathbb B^7}\setminus\{0\})$ and solves the PDE
					\begin{equation}
- (\delta^{ij} - \xi^{i}\xi^{j} ) \partial_i \partial_j u_1(\xi) - 9 \xi^{i} \partial_j u_1(\xi) =g(|\xi|) \label{dense range PDE}
					\end{equation}
				for $\xi\in\overline{\mathbb B_1^7}\setminus\{0\}$. Furthermore, direct calculations show that $\hat u_1(\r)=O(1)$ and $\hat u_1'(\r)=O(\r)$ for $\r$ near $0$. Thus, $u_1\in H^1(\mathbb B^7)$ and, consequently, $u_1$ is a weak solution of Equation \eqref{dense range PDE} on $\mathbb B^7$. By elliptic regularity, we infer that $u_1\in C^\infty_\text{rad}(\overline{\mathbb B^7})$. An application of the Lumer-Phillips Theorem now implies the first part of the claim. The proof of  second statement about the restriction properties is the same as in Lemma $3.5$ of \cite{CGS21}.
			\end{proof}

\bibliographystyle{plain}
\bibliography{bibfile}

\end{document}